\documentclass[11pt]{amsart}

\usepackage{amsmath}

\usepackage{amssymb}

\usepackage{amsfonts}

\usepackage{amsthm}

\usepackage{enumerate}

 \usepackage{esint}
 
 \usepackage{tikz}
 
 \usetikzlibrary{patterns}
 
 \usepackage{pgfplots}

 \usepackage{subcaption}
 \usepackage{graphicx}
 \usepackage{lipsum}

\usepackage{graphicx}

\usepackage{verbatim}

\newtheorem{innercustomthm}{Theorem}
\newenvironment{customthm}[1]
  {\renewcommand\theinnercustomthm{#1}\innercustomthm}
  {\endinnercustomthm}

  \newtheorem{innercustomlemma}{Lemma}
\newenvironment{customlemma}[1]
  {\renewcommand\theinnercustomlemma{#1}\innercustomlemma}
  {\endinnercustomlemma}

\newcommand{\wh}{\widehat}

\newcommand{\Norm}[1]{ \left\|  #1 \right\| }

\newcommand{\Be}{\begin{equation}}

\newcommand{\Ee}{\end{equation}}

\newcommand{\Bm}{\begin{multline}}

\newcommand{\Em}{\end{multline}}

\newcommand{\Bea}{\begin{eqnarray}}

\newcommand{\Eea}{\end{eqnarray}}

\newcommand{\Beas}{\begin{eqnarray*}}

\newcommand{\Eeas}{\end{eqnarray*}}

\newcommand{\Benu}{\begin{enumerate}}

\newcommand{\Eenu}{\end{enumerate}}

\newcommand{\Bi}{\begin{itemize}}

\newcommand{\Ei}{\end{itemize}}

\def\intslash{\fint}

\def\intslash{\rlap{\kern  .32em $\mspace {.5mu}\backslash$ }\int}

\def\qsl{{\rlap{\kern  .32em $\mspace {.5mu}\backslash$ }\int_{Q_x}}}

\def\Norm#1{{ \left\|  #1 \right\| }}

\def\emph#1{{\it #1 }}

\def\G{{\hbox{\bf G}}}

\font \roman = cmr10 at 10 true pt

\def\y{{\hbox{\roman y}}}

\def\be#1{\begin{equation}\label{ #1}}

\def\endeq{\end{equation}}

\def\endal{\end{align}}

\def\bas{\begin{align*}}

\def\eas{\end{align*}}

\def\bi{\begin{itemize}}

\def\ei{\end{itemize}}

\def\emph#1{{\it #1}}

\def\textbf#1{{\bf #1}}

\theoremstyle{plain}

   \newtheorem{theorem}{Theorem}[section]

   \newtheorem{conjecture}[theorem]{Conjecture}

   \newtheorem{proposition}[theorem]{Proposition}

   \newtheorem{lemma}[theorem]{Lemma}

   \newtheorem{corollary}[theorem]{Corollary}

   \newtheorem{theorem*}{Theorem}

\theoremstyle{remark}

   \newtheorem{remark}[theorem]{Remark}

\theoremstyle{definition}

\numberwithin{equation}{section}

\begin{document}
\title[Radial Fourier multipliers in $\mathbb{R}^3$ and $\mathbb{R}^4$]{Radial Fourier multipliers in $\mathbb{R}^3$ and $\mathbb{R}^4$}
\author[L. Cladek]{Laura Cladek}
\address{L. Cladek, Department of Mathematics\\ University of Wisconsin-Madison\\480 Lincoln Drive, Madison, WI 53706, USA}

\email{cladek@math.wisc.edu}


\thanks{The author would like to thank Andreas Seeger for helpful discussions and comments which improved the presentation of this paper. Part of this work was completed while the author was visiting the Institute for Computational and Experimental Research in Mathematics (ICERM), and the author is grateful to ICERM for its hospitality. Research supported in part by NSF Research and Training grant DMS 1147523}
\maketitle

\begin{abstract}
We prove that for radial Fourier multipliers $m: \mathbb{R}^3\to\mathbb{C}$ supported compactly away from the origin, $T_m$ is restricted strong type (p,p) if $K=\widehat{m}$ is in $L^p(\mathbb{R}^3)$, in the range $1<p<\frac{13}{12}$. We also prove an $L^p$ characterization for radial Fourier multipliers in four dimensions; namely, for radial Fourier multipliers $m: \mathbb{R}^4\to\mathbb{C}$ supported compactly away from the origin, $T_m$ is bounded on $L^p(\mathbb{R}^4)$ if and only if $K=\widehat{m}$ is in $L^p(\mathbb{R}^4)$, in the range $1<p<\frac{36}{29}$. Our method of proof relies on a geometric argument that exploits bounds on sizes of multiple intersections of three-dimensional annuli to control numbers of tangencies between pairs of annuli in three and four dimensions. \end{abstract}

\section{Introduction and statement of results}\label{intro}
In this paper we study radial multiplier transformations whose symbol is compactly supported away from the origin. These are operators $T_m$ defined via the Fourier transform by 
\begin{align*}
\mathcal{F}[{T_mf}](\xi)=m(\xi)\wh{f}(\xi),
\end{align*}
where $m:\mathbb{R}^d\to\mathbb{C}$ is a bounded, measurable, {radial} function supported in a compact subset of $\{\xi: 1/2<|\xi|<2\}$.
\newline
\indent
In the cases $p\ne 1, 2$, it is generally believed that is is impossible to give a reasonable characterization of all multiplier operators which are bounded on $L^p$. However, for \textit{radial} Fourier multipliers, a characterization can be obtained for an appropriate range of $p$. In \cite{hns}, Heo, Nazarov, and Seeger prove a strikingly simple characterization of radial multipliers that are bounded on $L^p(\mathbb{R}^d)$ in dimensions $d\ge 4$ for $1<p<\frac{2d-2}{d+1}$. 
\begin{customthm}{A}\label{hnsthm}
Let $d\ge 2$. If $m:\mathbb{R}^d\to\mathbb{C}$ is radial and supported in a compact subset of $\{\xi: 1/2<|\xi|<2\}$, the multiplier operator $T_m$ is bounded on $L^p(\mathbb{R}^d)$ if and only if the kernel $K=\widehat{m}$ is in $L^p(\mathbb{R}^d)$, in the range $1<p<\frac{2d-2}{d+1}$. \end{customthm}
The characterization in \cite{hns} was motivated by the earlier work \cite{gs} of Garrig\'{o}s and Seeger, where the authors obtained a similar characterization of all convolution operators with radial kernels acting on the space $L_{\text{rad}}^p$ of radial $L^p$ functions, in the larger range $1<p<\frac{2d}{d+1}$. 
\begin{customthm}{B}
Let $d\ge 2$. If $m:\mathbb{R}^d\to\mathbb{C}$ is radial and supported in a compact subset of $\{\xi: 1/2<|\xi|<2\}$, the multiplier operator $T_m$ is bounded on $L_{\text{rad}}^p(\mathbb{R}^d)$ if and only if the kernel $K=\widehat{m}$ is in $L^p(\mathbb{R}^d)$, in the range $1<p<\frac{2d}{d+1}$.
\end{customthm}
This range $1<p<\frac{2d}{d+1}$ is the optimal range for their result to hold, since for $p\ge \frac{2d}{d+1}$ one may construct radial kernels in $L^p$ that have Fourier transforms which are supported compactly away from the origin but which are also unbounded. By the same reasoning, the range $1<p<\frac{2d}{d+1}$ is also the largest possible range in which one could hope for the characterization from Theorem \ref{hnsthm} to hold. Thus one might propose the following conjecture, which we will refer to as the ``Radial Fourier Multiplier Conjecture."
\begin{conjecture}\label{conj}
Let $d\ge 2$. If $m:\mathbb{R}^d\to\mathbb{C}$ is radial and supported in a compact subset of $\{\xi: 1/2<|\xi|<2\}$, the multiplier operator $T_m$ is bounded on $L^p(\mathbb{R}^d)$ if and only if the kernel $K=\widehat{m}$ is in $L^p(\mathbb{R}^d)$, in the range $1<p<\frac{2d}{d+1}$.
\end{conjecture}
One can appreciate the strength of this conjecture by noting that since $\frac{2d}{d+1}$ is the critical value for the Bochner-Riesz Conjecture, the Bochner-Riesz Conjecture (and hence also the Restriction and Kakeya Conjectures) would follow as a special case from Conjecture \ref{conj}. However, the statement of Conjecture \ref{conj} is far more general than the Bochner-Riesz Conjecture, since it makes no a priori assumptions whatsoever on the regularity of the multiplier.
\newline
\indent
The arguments of \cite{hns} did not yield any results about radial Fourier multipliers in $\mathbb{R}^3$. In this paper, we will improve a key lemma of \cite{hns} in three dimensions to very nearly obtain a characterization of compactly supported radial Fourier multipliers $m$ bounded on $L^p(\mathbb{R}^3)$, in the range $1<p<\frac{13}{12}$.

\begin{theorem}\label{mainthm}
Let $m$ be a radial Fourier multiplier in $\mathbb{R}^3$ supported in $\{1/2<|\xi|<2\}$ and let $K=\mathcal{F}^{-1}[m]$. Then for $1<p<13/12$, if $K\in L^p$ then the multiplier operator $T_m$ is restricted strong type $(p, p)$, and moreover 
\begin{align*}
\Norm{K\ast f}_{L^p(\mathbb{R}^3)}\lesssim_p\Norm{K}_{L^p(\mathbb{R}^3)}\Norm{f}_{L^{p, 1}(\mathbb{R}^3)}.
\end{align*}
\end{theorem}
\begin{remark}
Our proof will also show that $\Norm{K\ast f}_{L^p}\lesssim_p\Norm{K}_{L^{p, 1}}\Norm{f}_{L^p}$, and we expect that $\Norm{K}_{L^{p, 1}}$ could be improved to $\Norm{K}_{L^p}$.
\end{remark}

We will also prove a full $L^p$ characterization for compactly supported radial Fourier multipliers in $\mathbb{R}^4$ in the range $1<p<\frac{36}{29}$, which improves on Heo, Nazarov, and Seeger's result.

\begin{theorem}\label{mainthm2}
Let $m$ be a radial Fourier multiplier in $\mathbb{R}^4$ supported in $\{1/2<|\xi|<2\}$ and let $K=\mathcal{F}^{-1}[m]$. Then for $1<p<36/29$, if $K\in L^p(\mathbb{R}^4)$, the multiplier operator $T_m$ is bounded on $L^p(\mathbb{R}^4)$, and moreover 
\begin{align*}
\Norm{K\ast f}_{L^p(\mathbb{R}^4)}\lesssim_p\Norm{K}_{L^p(\mathbb{R}^4)}\Norm{f}_{L^{p}(\mathbb{R}^4)}.
\end{align*}

\end{theorem}

Our proofs of Theorems \ref{mainthm} and \ref{mainthm2} refine the arguments of \cite{hns} while simultaneously incorporating new geometric input. A key divergence from the arguments of \cite{hns} is the exploitation of the underlying ``tensor product structure" inherent in the problem, a notion which will become clearer later. This, combined with a geometric argument involving sizes of multiple intersections of three-dimensional annuli, allows one to take advantage of improved scalar product estimates which were not used in \cite{hns}. However, since we exploit the tensor product structure of the problem, we are currently not able to deduce any local smoothing results for the wave equation as corollaries, as was able to be done in \cite{hns}.
\newline
\indent
The outline of the paper is as follows. The first portion of the paper will be devoted to the proof of Theorem \ref{mainthm}, which is less technical than the proof of Theorem \ref{mainthm2}. The second portion will give the proof of Theorem \ref{mainthm2}. At the end, we provide as an appendix the proof of the geometric lemma used in the proofs of both theorems.

\section{Preliminaries and reductions}\label{prelimsec}

In this section we will collect some necessary preliminary results and reductions. Versions of these results can be found in \cite{hns}, but we reproduce them here for completeness. In general, this section of the paper will very closely follow \cite{hns}, and for convenience we choose to adopt similar notation.

\subsection*{Discretization and density decomposition of sets}
The first step will be to discretize our problem, and in preparation for this we will first need to introduce some notation. Let $\mathcal{Y}$ be a $1$-separated set of points in $\mathbb{R}^3$ and let $\mathcal{R}$ be a $1$-separated set of radii $\ge 1$. Let $\mathcal{E}\subset\mathcal{Y}\times\mathcal{R}$ be a finite set that is also a \textit{product}, i.e. $\mathcal{E}=\mathcal{E}_Y\times\mathcal{E}_R$ where $\mathcal{E}_Y\subset\mathcal{Y}$ and $\mathcal{E}_R\subset\mathcal{R}$. (The assumption that $\mathcal{E}$ is a product was not used in \cite{hns}, but will be crucial for our argument.)
\newline
\indent
Let 
\begin{align*}
u\in\mathcal{U}=\{2^{\nu}, \nu=0, 1, 2,\ldots\}
\end{align*} 
be a collection of dyadic indices. For each $k$, let $\mathfrak{B}_k$ denote the collection of all $4$-dimensional balls of radius $\le 2^k$. For a ball $B$, let $\text{rad}(B)$ denote the radius of $B$. Following \cite{hns}, define:
\begin{align*}
\mathcal{R}_k:=\mathcal{R}\cap [2^k, 2^{k+1}),
\end{align*}
\begin{align*}
\mathcal{E}_k:=\mathcal{E}\cap(\mathcal{Y}\times\mathcal{R}_k),
\end{align*}
\begin{align*}
\wh{\mathcal{E}}_k(u):=\{(y, r)\in\mathcal{E}_k:\,\exists B\in\mathfrak{B}_k\text{ such that }\#(\mathcal{E}_k\cap B)\ge u\,\text{rad}(B)\},
\end{align*}
\begin{align*}
\mathcal{E}_k(u)=\wh{\mathcal{E}}_k(u)\setminus\bigcup_{\substack{u^{\prime}\in\mathcal{U}\\u^{\prime}>u}}\wh{\mathcal{E}}_k(u^{\prime}).
\end{align*}
We will refer to $u$ as the \textit{density} of the set $\mathcal{E}_k(u)$. Note that we have the decomposition
\begin{align*}
\mathcal{E}_k=\bigcup_{u\in\mathcal{U}}\mathcal{E}_k(u).
\end{align*}
Let $\sigma_r$ denote the surface measure on $rS^{2}$, the $2$-sphere centered at the origin of radius $r$. Now fix a smooth, radial function $\psi_0$ which is supported in the ball centered at the origin of radius $1/10$ such that $\wh{\psi_0}$ vanishes to order $40$ at the origin. Let $\psi=\psi_0\ast\psi_0$. For $y\in\mathcal{Y}$ and $r\in\mathcal{R}$, define
\begin{align*}
F_{y, r}=\sigma_r\ast\psi(\cdot-y).
\end{align*}
For a given function $c:\mathcal{Y}\times\mathcal{R}\to\mathbb{C}$, further define 
\begin{align*}
G_{u, k}:=\sum_{(y, r)\in\mathcal{E}_k(u)}c(y, r)F_{y, r},
\end{align*}
\begin{align*}
G_{u}:=\sum_{k\ge 0}G_{u, k},
\end{align*}
\begin{align*}
G_{k}:=\sum_{u\in\mathcal{U}}G_{u, k}.
\end{align*}

\subsection*{An interpolation lemma}
As a preliminary tool, we will need the following dyadic interpolation lemma.

\begin{lemma}\label{dyad}
Let $0<p_0<p_1<\infty$. Let $\{F_j\}_{j\in\mathbb{Z}}$ be a sequence of measurable functions on a measure space $\{\Omega, \mu\}$, and let $\{s_j\}$ be a sequence of nonnegative numbers. Assume that for all $j$, the inequality
\begin{align}\label{dl1}
\Norm{F_j}_{p_{\nu}}^{p_{\nu}}\le 2^{j p_{\nu}}M^{p_{\nu}}s_j
\end{align} 
holds for $\nu=0$ and $\nu=1$. Then for all $p\in (p_0, p_1)$, there is a constant $C=C(p_0, p_1, p)$ such that
\begin{align}\label{dl2}
\Norm{\sum_jF_j}_p^p\le C^pM^p\sum_j2^{jp}s_j.
\end{align}
\end{lemma}

\subsection*{The discretized $L^p$ inequality}

Our goal is to prove the following proposition, which we will see implies our main result for compactly supported multipliers.
\begin{proposition}\label{mainprop}
Let $\mathcal{E}$ and $\mathcal{E}_k$ be as above (recall that $\mathcal{E}$ has product structure). Let $c:\mathcal{E}\to\mathbb{C}$ be a function satisfying $|c(y, r)|\le 1$ for all $(y, r)\in\mathcal{E}$. Then for $1<p<13/12$,
\begin{align*}
\Norm{\sum_{(y, r)\in\mathcal{E}}c(y, r)F_{y, r}}_p^p\lesssim_p\sum_k 2^{2k}\#\mathcal{E}_k.
\end{align*}
\end{proposition}
Using the dyadic interpolation lemma (Lemma \ref{dyad}), we obtain the following corollary.
\begin{corollary}\label{maincor}
Let $E$ be any measurable set of finite measure, and $\chi_E$ its characteristic function. Suppose that $f$ is a measurable function satisfying $|f|\le\chi_E$. Then for $1<p<13/12$, we have
\begin{align}\label{dis}
\Norm{\sum_{(y, r)\in\mathcal{Y}\times\mathcal{R}}\gamma(r)f(y)F_{y, r}}_p\lesssim\bigg(\sum_{(y, r)\in\mathcal{Y}\times\mathcal{R}}|\gamma(r)\chi_E(y)|^pr^2\bigg)^{1/p}.
\end{align}
Also
\begin{align}\label{cont}
\Norm{\int_{\mathbb{R}^d}\int_1^{\infty}h(r)f(y)F_{y, r}\,dr\,dy}_p\lesssim\bigg(\int_{\mathbb{R}^d}\int_1^{\infty}|h(r)\chi_E(y)|^pr^2\,dr\,dy\bigg)^{1/p}.
\end{align}
\end{corollary}

\begin{proof}[Proof that Proposition \ref{mainprop} implies Corollary \ref{maincor}]
For $j\in\mathbb{Z}$, define the level sets
\begin{align*}
\mathcal{E}^j:=\{(y, r)\in\mathcal{Y}\times\mathcal{R}:\, 2^{j-1}<|\gamma(r)\chi_E(y)|\le 2^j\}. 
\end{align*}
Notice that $\mathcal{E}^j$ has product structure, so Proposition \ref{mainprop} implies that for $1<p<13/12$,
\begin{align*}
\Norm{\sum_{(y, r)\in\mathcal{E}^j}\gamma(r)f(y)F_{y, r}}_p^p\lesssim_p2^{jp}\sum_{(y, r)\in\mathcal{E}^j}r^2.
\end{align*}
Now apply Lemma \ref{dyad} with $F_j=\sum_{(y, r)\in\mathcal{E}^j}\gamma(r)f(y)F_{y, r}$, $M=1$, and $s_j=\sum_{(y, r)\in\mathcal{E}^j}r^2$ to obtain (\ref{dis}).
\newline
\indent
Now we prove (\ref{cont}). Let $y=z+w$ for $z\in\mathbb{Z}^3$ and $w\in Q_0:=[0, 1)^3$ and $r=n+\tau$ for $n\in\mathbb{N}$ and $0\le\tau<1$. By Minkowski's inequality and (\ref{dis}),
\begin{multline*}
\Norm{\int_{\mathbb{R}^d}\int_1^{\infty}h(r)f(y)F_{y, r}\,dr\,dy}_p
\\
\lesssim\int\int_{Q_0\times [0, 1)}\Norm{\sum_{z\in\mathbb{Z}^d}\sum_{n=1}^{\infty}h(n+\tau)f(z+w)F_{z+w, n+\tau}}_p\,dw\,d\tau
\\
\lesssim\int\int_{Q_0\times [0, 1)}\bigg(\sum_{z\in\mathbb{Z}^d}\sum_{n=1}^{\infty}|h(n+\tau)\chi_E(z+w)|^p(n+\tau)^2\bigg)^{1/p}\,dw\,d\tau
\\
\lesssim\bigg(\int_{\mathbb{R}^d}\int_1^{\infty}|h(r)\chi_E(y)|^pr^2\,dr\,dy\bigg)^{1/p},
\end{multline*}
where in the last step we have used H\"{o}lder's inequality.
\end{proof}

\subsection*{Support size estimates vs. $L^2$ inequalities}
As in \cite{hns}, we will show that the functions ${G_{u, k}}$ either have relatively small support size or satisfy relatively good $L^2$ bounds. We begin with a support size bound from \cite{hns} that improves as the density $u$ increases.
\begin{customlemma}{C}\label{L1lemma}
For all $u\in\mathcal{U}$, the Lebesgue measure of the support of $G_{u, k}$ is $\lesssim u^{-1}2^{2k}\#\mathcal{E}_k$.
\end{customlemma}
We will prove the following $L^2$ inequality which in some sense an improved version of Lemma $3.6$ from \cite{hns}, although the hypotheses are different since it is crucial that we assume that the underlying set $\mathcal{E}$ has product structure. This inequality improves as the density $u$ decreases. In \cite{hns}, the analogous $L^2$ inequality proved is 
\begin{align}\label{hnsl22}
\Norm{G_u}_2^2\lesssim u^{\frac{2}{d-1}}\log(2+u)\sum_k2^{k(d-1)}\#\mathcal{E}_k,
\end{align}
and when $d=3$ the term $u^{\frac{2}{d-1}}$ is equal to $u$. One may check that combining (\ref{hnsl22}) with Lemma \ref{L1lemma} as in the proof of Lemma \ref{Lplemma} below yields no result in three dimensions. We use geometric methods to improve on (\ref{hnsl22}) in three dimensions, and our argument will rely on Lemma \ref{geomlemma} proved later in Section \ref{geomsec}.

\begin{lemma}\label{L2lemma}
Let $\mathcal{E}$, $\mathcal{E}_k$, and $G_{u}$ be as above (recall that $\mathcal{E}$ has product structure). Assume $|c(y, r)|\le 1$ for $(y, r)\in\mathcal{Y}\times\mathcal{R}$. Then for every $\epsilon>0$,
\begin{align*}
\Norm{G_u}_2^2\lesssim_{\epsilon} u^{\frac{11}{13}+\epsilon}\sum_k2^{2k}\#\mathcal{E}_k.
\end{align*}
\end{lemma}
Combining Lemma \ref{L1lemma} and Lemma \ref{L2lemma}, we obtain the following $L^p$ bound.
\begin{lemma}\label{Lplemma}
For $p\le 2$, for every $\epsilon>0$,
\begin{align*}
\Norm{G_u}_p\lesssim_{\epsilon}u^{-(1/p-12/13-\epsilon)}(\sum_k2^{2k}\#\mathcal{E}_k)^{1/p}.
\end{align*}
\end{lemma}

\begin{proof}[Proof of Lemma \ref{Lplemma} given Lemma \ref{L1lemma} and Lemma \ref{L2lemma}]
By H\"{o}lder's inequality,
\begin{multline*}
\Norm{G_u}_p\lesssim (\text{meas}(\text{supp}(G_u)))^{1/p-1/2}\Norm{G_u}_2
\\
\lesssim_{\epsilon} u^{-1/p+1/2}u^{11/26+\epsilon}(\sum_k2^{2k}\#\mathcal{E}_k)^{1/p}
\\
\lesssim_{\epsilon}u^{12/13-1/p+\epsilon}(\sum_k2^{2k}\#\mathcal{E}_k)^{1/p}.
\end{multline*}
\end{proof}

Summing over $u\in\mathcal{U}$, we obtain Proposition \ref{mainprop}. Thus to prove Proposition \ref{mainprop} it suffices to prove Lemma \ref{L2lemma}.

\subsection*{Compactly supported multipliers}

Following \cite{hns}, we now show how one may deduce Theorem \ref{mainthm} from Corollary \ref{maincor}. Suppose that $m:\mathbb{R}^3\to\mathbb{C}$ is a bounded, measurable, radial function with compact support inside $\{\xi: 1/2<|\xi|<2\}$. Then $K=\mathcal{F}^{-1}[m]$ is radial, and so we may write $K(\cdot)=\kappa(|\cdot|)$ for some $\kappa:\mathbb{R}\to\mathbb{C}$. Fix a radial Schwartz function $\eta_0$ such that $\wh{\eta_0}(\xi)=1$ on $\text{supp}(m)$ and such that $\eta_0$ has Fourier support in $\{1/4<|\xi|<4\}$. Set $\eta=\mathcal{F}^{-1}[(\wh{\psi})^{-1}\wh{\eta_0}]$. We have $K\ast f=\eta\ast\psi\ast K\ast f$. Let $K_0=K\chi_{\{x:\,|x|\le 1\}}$ and write $K=K_0+K_{\infty}$. Since $\Norm{K_0}_1\lesssim\Norm{K}_p$, it suffices to show that the operator $f\mapsto\eta\ast\psi\ast K_{\infty}\ast f$ is restricted strong type $(p, p)$ with operator norm $\lesssim_p\Norm{K}_p$. Let $E$ be a measurable set of finite measure, and suppose that $|f|\le\chi_E$. 
We may write
\begin{align*}
\psi\ast K_{\infty}\ast f=\int_1^{\infty}\int\psi\ast\sigma_r(\cdot-y)\kappa(r)f(y)\,dy\,dr.
\end{align*}
By Corollary \ref{maincor}, we have
\begin{multline*}
\Norm{\eta\ast\psi\ast K_{\infty}\ast f}_p
\\
\lesssim_p\Norm{\psi\ast K_{\infty}\ast f}_p\lesssim_p\bigg(\int |\kappa(r)|^pr^{2}\,dr\bigg)^{1/p}\bigg(\int|\chi_E(y)|^p\,dy\bigg)^{1/p},
\end{multline*}
which implies the result of Theorem \ref{mainthm}.

\section{Proof of the $L^2$ inequality}
We have shown in the Section \ref{prelimsec} that to prove our main result Theorem \ref{mainthm} it remains to prove Lemma \ref{L2lemma}, and this section is dedicated to the proof of that lemma. The proof will rely on a geometric lemma about sizes of multiple intersections of three-dimensional annuli, which is stated and proved in Section \ref{geomsec}.

\subsection*{Estimates for scalar products}

In order to obtain the desired $L^2$ estimate, we need to examine pairwise interactions of the form $\left<F_{y, r}, F_{y^{\prime}, r^{\prime}}\right>$. By applying Plancherel's Theorem and writing $\wh{F_{y, r}}$ and $\wh{F_{y^{\prime}, r^{\prime}}}$ as expressions involving Bessel functions, the authors of \cite{hns} obtained the following estimates for $|\left<F_{y, r}, F_{y^{\prime}, r^{\prime}}\right>|$.

\begin{lemma}\label{bessel1}
For any choice of $r, r^{\prime}>1$ and $y, y^{\prime}\in\mathbb{R}^3$
\begin{align*}
|\left<F_{y, r}, F_{y^{\prime}, r^{\prime}}\right>|\lesssim\frac{(rr^{\prime})}{(1+|y-y^{\prime}|+|r-r^{\prime}|)}.
\end{align*}
\end{lemma}

The proof of this lemma used only the decay and not the oscillation of the Bessel functions. By exploiting the oscillation of the Bessel functions one may obtain the following improved bounds, which are crucial for our purposes. Since we will use this lemma in three and four dimensions, we state it in terms of dimension $d$, where the functions $F_{y, r}$ are defined analogously in $d$ dimensions as they are defined previously in three dimensions.

\begin{lemma}\label{bessel}
For any choice of $r, r^{\prime}>1$ and $y, y^{\prime}\in\mathbb{R}^d$ and any $N>0$,
\begin{multline*}
|\left<F_{y, r}, F_{y^{\prime}, r^{\prime}}\right>|\le C_N(rr^{\prime})^{\frac{d-1}{2}}(1+|y-y^{\prime}|+|r-r^{\prime}|)^{-\frac{d-1}{2}}
\\
\times\sum_{\pm, \pm}(1+|r\pm r^{\prime}\pm |y-y^{\prime}||)^{-N}.
\end{multline*}
\end{lemma}

\begin{proof}[Proof of Lemma \ref{bessel}]
We may write $\widehat{\sigma}_1$ in terms of Bessel functions as $\wh{\sigma}_1(\xi)=B_d(|\xi|)$, where
\begin{align*}
B_d(s)=c_ds^{-(d-2)/2}J_{(d-2)/2}
\end{align*}
and $J$ denotes the standard Bessel functions. This implies
\begin{align*}
\wh{\sigma}_r(\xi)=r^{d-1}B_d(r|\xi|).
\end{align*}
Since $\wh{\psi}$ is radial, we may write $\wh{\psi}(\xi)=a(|\xi|)$ for some rapidly decaying function $a$ that vanishes to high order (say $10d$) at the origin. By Plancherel, we have
\begin{multline*}
\left<F_{y, r}, F_{y', r'}\right>=\int\wh{\sigma}_r(\xi)\wh{\sigma}_{r'}(\xi)|\wh{\psi}(\xi)|^2e^{i\left<y'-y, \xi\right>}\,d\xi
\\
=c_d(rr')^{d-1}\int B_d(r\rho)B_d(r'\rho)B_d(|y-y'|\rho)|a(\rho)|^2\rho^{d-1}\,d\rho.
\end{multline*}
We will use the following well-known asymptotic expansion, which holds for $|x|\ge 1$ and any $M$:
\begin{align*}
B_d(x)=\sum_{\nu=0}^M(c_{\nu, k, d}^+e^{ix}+c_{\nu, k, d}^-e^{-ix})x^{-\nu-\frac{d-1}{2}}+x^{-M}E_{M, k, d}(x)
\end{align*}
where for any $k_1\ge 0$,
\begin{align*}
|E_{M, k, d}^{(k_1)}(x)|\le C(M, k, k_1, d).
\end{align*}
Using this expansion together with the higher order of vanishing of $a$ at the origin, one sees that there is a fixed Schwartz function $\eta$ so that we obtain for any $N>0$, 
\begin{multline*}
\left<F_{y, r}, F_{y', r'}\right>\lesssim (rr')^{(d-1)/2}(1+|y-y'|)^{-\frac{d-1}{2}}\sum_{\pm, \pm}\eta(r\pm r' \pm |y-y'|)
\\
+(1+|r-r'|+|y-y'|)^{-N}.
\end{multline*}
In fact, we may take $\eta$ to be the Fourier transform of $|a(\cdot)|^2\rho^{\alpha(d)}$ for some appropriate exponent $\alpha(d)$.
\end{proof}

\subsection*{Another preliminary reduction}
Recall that our goal is to estimate the $L^2$ norm of $G_u=\sum_{k\ge 0}G_{u, k}$. Let $N(u)$ be a sufficiently large number to be chosen later (it will be some harmless constant depending on $u$ that is essentially $O(\log(2+u))$). We split the sum in $k$ as $\sum_{k\le N(u)}G_{u, k}+\sum_{k>N(u)}G_{u, k}$ and apply Cauchy-Schwarz to obtain
\begin{align}\label{goal1}
\Norm{\sum_kG_{u, k}}_2^2\lesssim N(u)\bigg[\sum_k\Norm{G_{u, k}}_2^2+\sum_{k>k^{\prime}>N(u)}|\left<G_{u, k^{\prime}}, G_{u, k}\right>|\bigg].
\end{align}
We may thus separately estimate $\sum_k\Norm{G_{u, k}}_2^2$ and $\sum_{k>k^{\prime}>N(u)}|\left<G_{u, k^{\prime}}, G_{u, k}\right>|$, which divides the proof of the $L^2$ estimate into two cases, the first being the case of ``comparable radii" and the second being the case of ``incomparable radii."

\subsection*{Comparable radii}
We will first estimate $\sum_k\Norm{G_{u, k}}_2^2$. Our goal will be to prove the following lemma.

\begin{lemma}\label{lem1}
For every $\epsilon>0$, 
\begin{align}\label{comp}
\Norm{G_{u, k}}_2^2\lesssim_{\epsilon}2^{2k}(\#\mathcal{E}_k)u^{11/13+\epsilon}.
\end{align}
\end{lemma}

Fix $k$ and $u$. We first observe that for $(y, r), (y^{\prime}, r^{\prime})\in\mathcal{E}_k(u)$, we have $\left<F_{y, r}, F_{y^{\prime}, r^{\prime}}\right>=0$ unless $|(y, r)-(y^{\prime}, r^{\prime})|\le 2^{k+5}$. To estimate $\Norm{G_{u, k}}_2^2$ for a fixed $k$, we would thus like to bound
\begin{align*}
\sum_{\substack{(y, r), (y^{\prime}, r^{\prime})\in\mathcal{E}_k(u)\\ 2^{m}\le|(y, r)-(y^{\prime}, r^{\prime})|\le 2^{m+1}}}|\left<F_{y, r}, F_{y^{\prime}, r^{\prime}}\right>|
\end{align*}
for all $0\le m\le k+4$.
\newline
\indent 
Now fix $m\le k+4$. Let $\mathcal{Q}_{u, k, m}$ be a collection of almost disjoint cubes $Q\subset\mathbb{R}^4$ of sidelength $2^{m+5}$ such that $\mathcal{E}_k(u)\subset\bigcup_{Q\in\mathcal{Q}_{u, k, m}}Q$ and so that every $Q$ has nonempty intersection with $\mathcal{E}_k(u)$. Let $Q^{\ast}$ denote the $2^{5}$-dilate of $Q$ and $\mathcal{Q}_{u, k, m}^{\ast}$ the corresponding collection of dilated cubes. Observe that
\begin{multline}\label{rand1}
\Norm{G_{u, k}}_2^2\lesssim\sum_{0\le m\le k+4}\bigg(\sum_{\substack{(y, r), (y^{\prime}, r^{\prime})\in\mathcal{E}_k(u)\\ 2^{m}\le|(y, r)-(y^{\prime}, r^{\prime})|\le 2^{m+1}}}|\left<F_{y, r}, F_{y^{\prime}, r^{\prime}}\right>|
\bigg)\\+\sum_{(y, r)\in\mathcal{E}_k(u)}\Norm{F_{y, r}}_2^2
\\
\lesssim\sum_{0\le m\le k+4}\bigg(\sum_{Q\in\mathcal{Q}_{u, k, m}}\bigg(\sum_{\substack{(y, r), (y^{\prime}, r^{\prime})\in(\mathcal{E}_k(u)\cap Q^{\ast})\\2^{m}\le|(y, r)-(y^{\prime}, r^{\prime})|\le 2^{m+1}}}|\left<F_{y, r}, F_{y^{\prime}, r^{\prime}}\right>|\bigg)\bigg)
\\
+\sum_{(y, r)\in\mathcal{E}_k(u)}\Norm{F_{y, r}}_2^2.
\end{multline} 
\indent 
Now we introduce some terminology that will be useful. For a subset $\mathcal{S}\subset\mathcal{Y}\times\mathcal{R}$, define its $\mathcal{Y}$ and $\mathcal{R}$-projections by 
\begin{align*}
\mathcal{S}_Y=\{y\in\mathcal{Y}:\,\exists\,(y, r)\in\mathcal{S}\}
\end{align*} 
and 
\begin{align*}
\mathcal{S}_R=\{r\in\mathcal{R}:\,\exists\, (y, r)\in\mathcal{S}\}.
\end{align*} 
Also define the \textit{product-extension} $\mathcal{S}^{\times}$ of $\mathcal{S}\subset\mathcal{Y}\times\mathcal{R}$ to be the set $\mathcal{S}_{Y}\times\mathcal{S}_R$. We also define some parameters associated with a fixed $Q\in\mathcal{Q}_{u, k, m}$. Let $N_{R, Q}$ be the cardinality of the $\mathcal{R}$-projection of $\mathcal{E}_k\cap Q^{\ast}$, i.e.
\begin{align*}
N_{R, Q}:=\#((\mathcal{E}_k\cap Q^{\ast})_R)=\#\{r: \exists (y, r)\in\mathcal{E}_k\cap Q^{\ast}\}.
\end{align*}
Similarly define
\begin{align*}
N_{Y, Q}:=\#((\mathcal{E}_k\cap Q^{\ast})_Y)=\#\{y: \exists (y, r)\in\mathcal{E}_k\cap Q^{\ast}\}.
\end{align*}
We also note the following important observation which we will use repeatedly. Using the definition of the sets $\mathcal{E}_k(u)$ and the fact that $\mathcal{E}_k$ has product structure, one may see that if $Q\in\mathcal{Q}_{u, k, m}$ is such that $(\mathcal{E}_k(u)\cap Q^{\ast})$ is nonempty, then 
\begin{align}\label{ess}
|N_{Y, Q}\cdot N_{R, Q}|\lesssim |\mathcal{E}_k\cap Q^{\ast}|\lesssim u2^m.
\end{align}
We remark that the product structure of the sets $\mathcal{E}_k$ is related to the ``tensor product structure" intrinsic to radial Fourier multipliers, mentioned in Section \ref{intro}. Now with (\ref{rand1}) in mind, we will prove the following lemma. 

\begin{lemma}\label{rand2}
 For each $Q\in\mathcal{Q}_{u, k, m}$, we have the estimates
\begin{multline}\label{rand3}
\sum_{\substack{(y, r), (y^{\prime}, r^{\prime})\in(\mathcal{E}_k(u)\cap Q^{\ast})\\2^{m}\le|(y, r)-(y^{\prime}, r^{\prime})|\le 2^{m+1}}}|\left<F_{y, r}, F_{y^{\prime}, r^{\prime}}\right>|
\\
\lesssim N_{R, Q}(\#(\mathcal{E}_k\cap Q^{\ast}))2^{2(k-m/2)}(m\log(u))\max(u^{5/6}2^{5m/6}, u2^{m/2})
\end{multline}
and
\begin{multline}\label{rand4}
\sum_{\substack{(y, r), (y^{\prime}, r^{\prime})\in(\mathcal{E}_k(u)\cap Q^{\ast})\\2^{m}\le|(y, r)-(y^{\prime}, r^{\prime})|\le 2^{m+1}}}|\left<F_{y, r}, F_{y^{\prime}, r^{\prime}}\right>|
\\
\lesssim 2^{2(k-m/2)}(\#(\mathcal{E}_k\cap Q^{\ast}))u2^m(N_{R, Q})^{-1}.
\end{multline}
\end{lemma}
We will then choose the better estimate from Lemma \ref{rand2} depending on $N_{R, Q}$ and sum over all $Q\in\mathcal{Q}_{u, k, m}$ and then over all $m\ge u^a$ where $a$ is a number to be chosen later. We will then use other methods to deal with the case $m\le u^a$, from which we will then obtain Lemma \ref{lem1}.

\begin{proof}[Proof of Lemma \ref{rand2}]
We will first prove (\ref{rand3}). By incurring a factor of $N_{R, Q}^2$, to estimate $\sum_{\substack{(y, r), (y^{\prime}, r^{\prime})\in(\mathcal{E}_k(u)\cap Q^{\ast})\\2^{m}\le|(y, r)-(y^{\prime}, r^{\prime})|\le 2^{m+1}}}|\left<F_{y, r}, F_{y^{\prime}, r^{\prime}}\right>|$ it suffices to estimate for a fixed pair $r_1, r_2$
\begin{align*}
\sum_{\substack{(y, r_1), (y^{\prime}, r_2)\in(\mathcal{E}_k(u)^{\times}\cap Q^{\ast})\\2^{m}\le|(y, r_1)-(y^{\prime}, r_2)|\le 2^{m+1}}}|\left<F_{y, r_1}, F_{y^{\prime}, r_2}\right>|,
\end{align*}
i.e. to restrict $(y, r)$ and $(y^{\prime}, r^{\prime})$ to lie in fixed rows of the product-extension of $\mathcal{E}_{k}(u)\cap Q^{\ast}$. (Our estimates will not depend on the particular choice of $r_1$ and $r_2$.)
\newline
\indent
Now, referring to the estimate in Lemma \ref{bessel}, we see that for a fixed $y, r_1, r_2$ we have that $|\left<F_{y, r_1}, F_{y^{\prime}, r_2}\right>|$ decays rapidly as $y^{\prime}$ moves away from the set $\{y^{\prime}:\,|y-y^{\prime}|=|r_1-r_2|\text{ or }|y-y^{\prime}|=r_1+r_2\}$, which is contained in a union of two annuli of thickness $2$ and radii $|r_1-r_2|$ and $r_1+r_2$ centered at $y$. 
\newline
\indent
Let $s\ge 0$, fix $t\le 2^{m+10}$, and define $K_k(Q, s, t)$ to be the number of points $y\in(\mathcal{E}_k(u)\cap Q^{\ast})_Y$ such that there are $\ge 2^s$ many points $y^{\prime}\in(\mathcal{E}_k\cap Q^{\ast})_Y$ such that $y^{\prime}$ lies in the annulus of inner radius $t$ and thickness $3$ centered at $y$. That is, define
\begin{align*}
K_k(Q, s, t):=\#\{y\in(\mathcal{E}_k(u)\cap Q^{\ast})_Y: \text{there exists at least }2^s\text{ many points }
\\
y^{\prime}\in(\mathcal{E}_k\cap Q^{\ast})_Y\text{ such that }||y^{\prime}-y|-(t+1.5)|\le 1.5\}.
\end{align*}
In view of the observation in the previous paragraph, for a given $s$ and a fixed number $t\le 2^{m+10}$, we would like to prove a bound on $K_k(Q, s, t)$. Our bound will depend on $s$ and $m$ but be independent of the choice of $t\le 2^{m+10}$. For this reason, we define the quantity
\begin{align*}
K_k^{\ast}(Q, s):=\max_{0 \le t\le 2^{m+10}}K_k(Q, s, t),
\end{align*}
and we will see that $K_k^{\ast}(Q, s)$ satisfies the same bound we prove for $K_k(Q, s, t)$. Our bound for $K_k(Q, s, t)$ will decay as $2^s$ gets larger and closer to $N_{Y, Q}$; in other words, ``most" of the points $y$ in $(\mathcal{E}_k(u)\cap Q^{\ast})_Y$ cannot have a large proportion of other points in $(\mathcal{E}_k\cap Q^{\ast})_Y$ lie in the annulus of inner radius $t$ and thickness $3$ centered at $y$. If we take $t=|r_1-r_2|$ or $t=r_1+r_2$, we see that this implies that ``most" of the $F_{y, r}$ with $(y, r)\in (\mathcal{E}_k(u)\cap Q^{\ast})_Y\times\{r_1\}$ do not ``interact badly" (where by badly we mean to the worst possible extent allowed by Lemma \ref{bessel}, i.e. internal tangencies of annuli) with most of the other $F_{y^{\prime}, r^{\prime}}$ where $(y^{\prime}, r^{\prime})\in(\mathcal{E}_k\cap Q^{\ast})_Y\times\{r_2\}$. This will allow us to obtain (\ref{rand3}), which is a good estimate in the case that $N_{R, Q}$ is small.
\newline
\indent
More precisely, we will prove
\begin{align}\label{Kest}
K_k^{\ast}(Q, s)\lesssim\max[u2^mN_{Y, Q}^{5/3}2^{-2s}, u2^{m/2}N_{Y, Q}2^{-s}].
\end{align}
Combining this with the trivial bound $K_k^{\ast}(Q, s)\lesssim N_{Y, Q}$ yields
\begin{align}\label{Kest2}
K_k^{\ast}(Q, s)\lesssim\max[\min(u2^mN_{Y, Q}^{5/3}2^{-2s}, N_{Y, Q}), \min(u2^{m/2}N_{Y, Q}2^{-s}, N_{Y, Q})].
\end{align}
Note that (\ref{Kest}) gives decay in the number of points $K_k^{\ast}(Q, s)$ (i.e. $K_k^{\ast}(Q, s)\ll N_{Y, Q}$) if we have that
\begin{enumerate}
\item$N_{Y, Q}^{5/3}2^{-2s}u2^m\ll N_{Y, Q}$, i.e. if $2^s\gg N_{Y, Q}^{1/3}u^{1/2}2^{m/2}$, and also
\item  $N_{Y, Q}2^{-s}u2^{m/2}\ll N_{Y, Q}$, i.e. if $2^s\gg u2^{m/2}$.
\end{enumerate}
Using Lemma \ref{bessel}, we may bound
\begin{multline}\label{b1}
\sum_{\substack{(y, r), (y^{\prime}, r^{\prime})\in(\mathcal{E}_k(u)\cap Q^{\ast})\\2^{m}\le|(y, r)-(y^{\prime}, r^{\prime})|\le 2^{m+1}}}|\left<F_{y, r}, F_{y^{\prime}, r^{\prime}}\right>|
\\
\lesssim \sum_{r_1, r_2\in (\mathcal{E}_k(u)\cap Q^{\ast})_R}\bigg(\sum_{\substack{y, y^{\prime}\in(\mathcal{E}_k(u)\cap Q^{\ast})_Y\\2^{m}\le|(y, r_1)-(y^{\prime}, r_2)|\le 2^{m+1}}}|\left<F_{y, r_1}, F_{y^{\prime}, r_2}\right>|\bigg)
\\
\lesssim 2^{2(k-m/2)}\sum_{r_1, r_2\in (\mathcal{E}_k(u)\cap Q^{\ast})_R}\bigg(\sum_{0\le a\le m+10}\bigg(\sum_{y\in(\mathcal{E}_k(u)\cap Q^{\ast})_Y}
\\
\sum_{\substack{y^{\prime}\in(\mathcal{E}_k(u)\cap Q^{\ast})_Y:\\min_{\pm, \pm}(1+|r_1\pm r_2\pm |y-y^{\prime}||)\approx 2^a}}
2^{-aN}
\bigg)\bigg)
\\
\lesssim 2^{2(k-m/2)}\sum_{r_1, r_2\in (\mathcal{E}_k(u)\cap Q^{\ast})_R}\bigg(\sum_{0\le a\le m+10}2^{-aN}
\\
\times\bigg(\sum_{s\ge 0: 2^s\le 2N_{Y, Q}}K_k^{\ast}(Q, s)2^s\bigg)\bigg)
\\
\lesssim 2^{2(k-m/2)}N_{Q, R}^2\sum_{s\ge 0: 2^s\le 2N_{Y, Q}}K_k^{\ast}(Q, s)2^s.
\end{multline}
Assuming (\ref{Kest2}) holds, we have
\begin{multline}\label{longy1}
\sum_{\substack{(y, r), (y^{\prime}, r^{\prime})\in(\mathcal{E}_k(u)\cap Q^{\ast})\\2^{m}\le|(y, r)-(y^{\prime}, r^{\prime})|\le 2^{m+1}}}|\left<F_{y, r}, F_{y^{\prime}, r^{\prime}}\right>|
\\
\lesssim N_{R, Q}^22^{2(k-m/2)}\sum_{s\ge 0: 2^s\lesssim N_{Y, Q}}\max[\min(u2^mN_{Y, Q}^{5/3}2^{-s}, N_{Y, Q}2^s), 
\\
\min(u2^{m/2}N_{Y, Q}, N_{Y, Q}2^s)]
\\
\lesssim N_{R, Q}^22^{2(k-m/2)}\max\bigg\{\sum_{s\ge 0: 2^s\lesssim N_{Y, Q}}\min(u2^mN_{Y, Q}^{5/3}2^{-s}, N_{Y, Q}2^s); 
\\
\sum_{s\ge 0: 2^s\lesssim N_{Y, Q}}\min(u2^{m/2}N_{Y, Q}, N_{Y, Q}2^s)\bigg\}
\end{multline}
Now, note that $u2^mN_{Y, Q}^{5/3}2^{-s}\ge N_{Y, Q}2^s$ if and only if $2^s\le u^{1/2}2^{m/2}N_{Y, Q}^{1/3}$. Thus choosing the better estimate in the term $\min(u2^mN_{Y, Q}^{5/3}2^{-s}, N_{Y, Q}2^s)$ depending on $s$ yields that 
\begin{align*}
\sum_{s\ge 0: 2^s\lesssim N_{Y, Q}}\min(u2^mN_{Y, Q}^{5/3}2^{-s}, N_{Y, Q}2^s)\lesssim u^{1/2}2^{m/2}N_{Y, Q}^{4/3}.
\end{align*}
Note that $u2^{m/2}N_{Y, Q}\ge N_{Y, Q}2^s$ if and only if $2^s\le u2^{m/2}$. Thus choosing the better estimate in the term $\min(u2^{m/2}N_{Y, Q}, N_{Y, Q}2^s)$ depending on $s$ yields that
\begin{align*}
\sum_{s\ge 0: 2^s\lesssim N_{Y, Q}}\min(u2^{m/2}N_{Y, Q}, N_{Y, Q}2^s)\lesssim \log(N_{Y, Q})N_{Y, Q}\,u2^{m/2}.
\end{align*}
It follows that the left hand side of (\ref{longy1}) is bounded by
\begin{multline}
N_{R, Q}^22^{2(k-m/2)}N_{Y, Q}\log(N_{Y, Q})\max(N_{Y, Q}^{1/3}u^{1/2}2^{m/2}, u2^{m/2})
\\
\lesssim N_{R, Q}^22^{2(k-m/2)}N_{Y, Q}(m\log(u))\max(u^{5/6}2^{5m/6}, u2^{m/2})
\\
\lesssim N_{R, Q}(\#(\mathcal{E}_k\cap Q^{\ast}))2^{2(k-m/2)}(m\log(u))\max(u^{5/6}2^{5m/6}, u2^{m/2}),
\end{multline}
which proves (\ref{rand3}). This will be a good estimate when $N_{R, Q}$ is small. 
\newline
\indent
Thus to prove (\ref{rand3}) it remains to prove (\ref{Kest}). We will in fact prove (\ref{Kest}) with $K_k^{\ast}(Q, s)$ replaced by $K_k(Q, s, t)$, uniformly in $t\le 2^{m+10}$. Fix $t\le 2^{m+10}$ and let $j=\lceil{\log_2(t)}\rceil$ and cover $(\mathcal{E}_k(u)\cap Q^{\ast})_Y$ by $\lesssim 2^{3(m-j)}$ many $3$-dimensional almost disjoint balls of radius $2^{j+5}$; denote this collection of balls as $\mathfrak{B}=\{B_i\}$. For each $i$, we define a collection of ``special" points $A_i(Q, s, t)$ to be the set of all points $y\in(\mathcal{E}_k(u)\cap Q^{\ast})_Y\cap B_i$ such that there are $\ge 2^s$ many points $y^{\prime}\in(\mathcal{E}_k\cap Q^{\ast})_Y$ such that $y^{\prime}$ lies in the annulus of radius $t$ and thickness $3$ centered at $y$. That is, we define
\begin{align*}
A_{k, i}(Q, s, t):=\{y\in(\mathcal{E}_k(u)\cap Q^{\ast})_Y\cap B_i: \text{ there exist at least }2^s\text{ many points }
\\
y^{\prime}\in(\mathcal{E}_k\cap Q^{\ast})_Y\text{ such that }||y^{\prime}-y|-(t+1.5)|\le 1.5\}.
\end{align*}
Let $K_{k, i}(Q, s, t)$ denote the cardinality of $A_{k, i}(Q, s, t)$. Now cover each $B_i$ with $\lesssim 2^{3(j-l)}$ many almost disjoint $3$-dimensional balls $B_{i, j}$ of radius $2^l$ for some $l\le j$. Each such ball contains at most $u2^l$ many points of $A_{k, i}(Q, s, t)$, so for a fixed $i$ there must be at least $\gtrsim K_{k, i}(Q, s, t)(u2^l)^{-1}$ many balls $B_{i, j}$ that contain at least one point in $A_{k, i}(Q, s, t)$. Thus there must be at least $\gtrsim K_{k, i}(Q, s, t)(u2^l)^{-1}$ many such points in $B_i\cap A_{k, i}(Q, s, t)$ spaced apart by $\gtrsim 2^l$; call this set $D_{k, i}(Q, s, t)$. But by Lemma \ref{geomlemma}, which we prove later in Section \ref{geomsec} of the paper, the size of three-fold intersections of annuli of radius $t\approx 2^j$ and thickness $3$ spaced apart by $\approx 2^l$ with centers lying a ball of radius $2^{j-5}$ is bounded above by $2^{3(j-l)}$ provided that $l\ge j/2+20$. 
\newline
\indent
It follows that if $l\ge j/2+20$, then for each of these $\approx K_{k, i}(Q, s, t)(u2^l)^{-1}$ many points $p\in D_{k, i}(Q, s, t)$, there can be at most 
\begin{align*}
\lesssim K_{k, i}(Q, s, t)^2(u2^l)^{-2}2^{3(j-l)}
\end{align*}
points lying inside the $t$-annulus centered at $p$ that are simultaneously contained in at least two other different $t$-annuli centered at points in $D_{k, i}(Q, s, t)$. This implies that if $N_{Y, Q, i}$ denotes the cardinality of $(\mathcal{E}_k\cap Q^{\ast})_Y\cap B_i^{\ast}$ where $B_i^{\ast}=10B_i$, then we have
\begin{align}\label{c1}
N_{Y, Q, i}\gtrsim K_{k, i}(Q, s, t)(u2^l)^{-1}2^s,
\end{align}
which is essentially $2^s$ times the number of points in $D_{k, i}(Q, s, t)$, provided that $2^s$ is much bigger than the total number of points lying inside a $t$-annulus centered at $p$ that are simultaneously contained in at least two other different $t$-annuli centered at points in $D_{k, i}(Q, s, t)$, i.e. provided that
\begin{align}\label{c2}
K_{k, i}(Q, s, t)^2(u2^l)^{-2}2^{3(j-l)}\ll 2^s
\end{align}
and 
\begin{align*}
l\ge j/2+20.
\end{align*} 
Solving for $2^l$ in (\ref{c2}) yields
\begin{align}\label{c3}
2^l\gg K_{k, i}(Q, s, t)^{2/5}2^{3j/5}u^{-2/5}2^{-s/5}.
\end{align}
Thus choosing a minimal $l$ such that 
\begin{align*}
2^l\gg\max(K_{k, i}(Q, s, t)^{2/5}2^{3j/5}u^{-2/5}2^{-s/5}, 2^{j/2})
\end{align*}
for a sufficiently large implied constant and substituting into (\ref{c1}) yields
\begin{align}\label{c4}
K_{k, i}(Q, s, t)\lesssim\max[u2^mN_{Y, Q, i}^{5/3}2^{-2s}, u2^{m/2}N_{Y, Q, i}2^{-s}],
\end{align}
and summing over all $i$ and using the almost-disjointness of the $B_i^{\ast}$ gives
\begin{align}\label{c6}
K_{k}(Q, s, t)\lesssim\max[u2^mN_{Y, Q}^{5/3}2^{-2s}, u2^{m/2}N_{Y, Q}2^{-s}].
\end{align}
Taking the maximum over all $0\le t\le 2^{m+10}$ proves (\ref{Kest}) and hence also (\ref{rand3}).
\newline
\indent
It remains to prove (\ref{rand4}), which will be a good estimate in the case that $N_{R, Q}$ is large. For a fixed $(y, r)\in Q^{\ast}$ and a fixed $y^{\prime}\in(\mathcal{E}_k(u)\cap Q^{\ast})_Y$, there are at most two values of $r^{\prime}$ away from which $\left<F_{y, r}, F_{y^{\prime}, r^{\prime}}\right>$ decays rapidly. Thus using Lemma \ref{bessel} we may estimate
\begin{multline}\label{b2}
\sum_{\substack{(y, r), (y^{\prime}, r^{\prime})\in(\mathcal{E}_k(u)\cap Q^{\ast})\\2^m\le |(y, r)-(y^{\prime}, r^{\prime})|\le 2^{m+1}}}|\left<F_{y, r}, F_{y^{\prime}, r^{\prime}}\right>|
\\
\lesssim\sum_{0\le a\le m+10}\bigg(\sum_{(y, r)\in(\mathcal{E}_k(u)\cap Q^{\ast})}\bigg(\sum_{y^{\prime}\in(\mathcal{E}_k(u)\cap Q^{\ast})_Y}
\\
\bigg(\sum_{\substack{r^{\prime}\in(\mathcal{E}_k(u)\cap Q^{\ast})_R\\2^m\le |(y, r)-(y^{\prime}, r^{\prime})|\le 2^{m+1}\\ \min_{\pm, \pm}(1+|r\pm r^{\prime}\pm |y-y^{\prime}||)\approx 2^a}}2^{-Na}2^{2(k-m/2)}\bigg)\bigg)\bigg)
\\
\lesssim 2^{2(k-m/2)}(\#(\mathcal{E}_k(u)\cap Q^{\ast}))N_{Y, Q}
\\
\lesssim 2^{2(k-m/2)}(\#(\mathcal{E}_k(u)\cap Q^{\ast}))u2^m(N_{Q, R})^{-1},
\end{multline}
and the proof of (\ref{rand4}) is complete.
\end{proof}

We will now use Lemma \ref{rand2} to prove Lemma \ref{lem1}.

\begin{proof}[Proof of Lemma \ref{lem1}]
Fix an $a>0$ to be determined later. As in \cite{hns}, we split $G_k=\sum_{\mu}G_{k, \mu}$, where for each positive integer $\mu$ we set
\begin{align*}
I_{k, \mu}=[2^k+(\mu-1)u^a, 2^k+\mu u^a),
\end{align*}
\begin{align*}
\mathcal{E}_{k, \mu}=\mathcal{E}_k\cap (\mathcal{Y}\times I_{k, \mu}),
\end{align*}
\begin{align*}
G_{k, \mu}=\sum_{(y, r)\in\mathcal{E}_{k, \mu}}c(y, r)F_{y, r},
\end{align*}
and
\begin{align*}
G_{k, \mu, r}=\sum_{y:\,(y, r)\in\mathcal{E}_k}c(y, r)F_{y, r}.
\end{align*}
We have
\begin{align}\label{split2}
\Norm{G_k}_2^2\lesssim\Norm{\sum_{\mu}G_{k, \mu}}_2^2\lesssim\sum_{\mu}\Norm{G_{k, \mu}}_2^2+\sum_{\mu^{\prime}>\mu+10}|\left<G_{k, \mu^{\prime}}, G_{k, \mu}\right>|.
\end{align}
By Cauchy-Schwarz,
\begin{align*}
\Norm{G_{k, \mu}}_2^2\lesssim u^a\sum_{r\in\mathcal{I}_{k, \mu}\cap\mathcal{R}}\Norm{G_{k, \mu, r}}_2^2.
\end{align*}
Write
\begin{align*}
G_{k, \mu, r}=\bigg(\sum_{y:\,(y, r)\in\mathcal{E}_{k, \mu}}c(y, r)\psi_0(\cdot-y)\bigg)\ast(\sigma_r\ast\psi_0).
\end{align*}
By the Fourier decay of $\sigma_r$ and the order of vanishing of $\psi_0$ at the origin, we have
\begin{align*}
\Norm{\widehat{\sigma_r}\widehat{\psi}_0}_{\infty}\lesssim r.
\end{align*}
Since the square of the $L^2$ norm of $\sum_{y:\,(y, r)\in\mathcal{E}_{k, \mu}}c(y, r)\psi_0(\cdot -y)$ is $\lesssim\#\{y\in\mathcal{Y}:\,(y, r)\in\mathcal{E}_{k, \mu}\}$, we have
\begin{align}\label{aest}
\sum_{\mu}\Norm{G_{k, \mu}}_2^2\lesssim u^a\sum_{\mu}\sum_{r\in\mathcal{I}_{k, \mu}\cap\mathcal{R}}\Norm{G_{k, \mu, r}}_2^2\lesssim u^a2^{2k}\#\mathcal{E}_k.
\end{align}
By (\ref{split2}), it remains to estimate $\sum_{\mu^{\prime}>\mu+10}|\left<G_{k, \mu^{\prime}}, G_{k, \mu}\right>|$.
\newline
\indent
Fix $\epsilon>0$. We will use (\ref{rand3}) when $N_{R, Q}\le 2^{m\epsilon}\min(u^{1/12+a/12}, u^{a/4})$ and (\ref{rand4}) when $N_{R, Q}\ge 2^{m\epsilon}\min(u^{1/12+a/12}, u^{a/4})$. We write 
\begin{multline*}
\sum_{\substack{(y, r), (y^{\prime}, r^{\prime})\in\mathcal{E}_k(u)\\|(y, r)-(y^{\prime}, r^{\prime})|\ge u^a}}|\left<F_{y, r}, F_{y^{\prime}, r^{\prime}}\right>|
\\
\lesssim\sum_{m:\,2^m\ge u^a}\bigg(\sum_{\substack{(y, r), (y^{\prime}, r^{\prime})\in\mathcal{E}_k(u)\\|(y, r)-(y^{\prime}, r^{\prime})|\approx 2^m}}\bigg(\sum_{\substack{Q\in\mathcal{Q}_{u, k, m}\\ N_{R, Q}\le 2^{m\epsilon}\min(u^{1/12+a/12}, u^{1/4})}}|\left<F_{y, r}, F_{y^{\prime}, r^{\prime}}\right>|
\\
+\sum_{\substack{Q\in\mathcal{Q}_{u, k, m}\\ N_{R, Q}\ge 2^{m\epsilon}\min(u^{1/12+a/12}, u^{1/4})}}|\left<F_{y, r}, F_{y^{\prime}, r^{\prime}}\right>|\bigg)\bigg).
\end{multline*}
One sees that
\begin{align}\label{onetwo}
\sum_{\substack{(y, r), (y^{\prime}, r^{\prime})\in\mathcal{E}_k(u)\\|(y, r)-(y^{\prime}, r^{\prime})|\ge u^a}}|\left<F_{y, r}, F_{y^{\prime}, r^{\prime}}\right>|\lesssim I+II,
\end{align}
where using (\ref{rand3}) when $N_{R, Q}\le 2^{m\epsilon}\min(u^{1/12+a/12}, u^{a/4})$ and summing over all $Q\in\mathcal{Q}_{u, k, m}$ and over all $m$ such that $2^m\ge u^a$ we have
\begin{multline}\label{b3}
I:= 2^{2k}(\#\mathcal{E}_k)\log(u)
\\
\times\sum_{m: 2^m\ge u^a} u^{\epsilon}\max\bigg\{2^{-m/6+\epsilon}\min(u^{11/12+a/12}, u^{5/6+a/4}), \\2^{-m/2+\epsilon}\min(u^{13/12+a/12}, u^{1+a/4})\bigg\}
\\
\lesssim
 2^{2k}(\#\mathcal{E}_k)u^{\epsilon}\max\bigg\{u^{-a/6}\min(u^{11/12+a/12}, u^{5/6+a/4}),
 \\
 u^{-a/2}\min(u^{13/12+a/12}, u^{1+a/4})\bigg\},
\end{multline}
and using (\ref{rand4}) when $N_{R, Q}\ge 2^{m\epsilon}\min(u^{1/12+a/12}, u^{a/4})$ and summing over all $Q$ and over all $m$ such that $2^m\ge u^a$ we have
\begin{multline}\label{b4}
II:=2^{2k}(\#\mathcal{E}_k)u^{\epsilon}\sum_{m: 2^m\ge u^a}2^{-m\epsilon}\max(u^{11/12-a/12}, u^{1-a/4})
\\
\lesssim_{\epsilon}2^{2k}(\#\mathcal{E}_k)u^{\epsilon}\max(u^{11/12-a/12}, u^{1-a/4}).
\end{multline}
Combining (\ref{split2}), (\ref{aest}) and (\ref{onetwo}), we thus have the estimate
\begin{multline*}
\Norm{G_{u, k}}_2^2\lesssim_{\epsilon}2^{2k}(\#\mathcal{E}_k)\bigg[u^a+u^{\epsilon}\max\bigg\{u^{-a/6}\min(u^{11/12+a/12}, u^{5/6+a/4}),
 \\
 u^{-a/2}\min(u^{13/12+a/12}, u^{1+a/4})\bigg\}+u^{\epsilon}\max(u^{11/12-a/12}, u^{1-a/4})\bigg].
\end{multline*}
Choose $a=11/13$ to obtain
\begin{align*}
\Norm{G_{u, k}}_2^2\lesssim_{\epsilon}2^{2k}(\#\mathcal{E}_k)u^{11/13+\epsilon}
\end{align*}
for every $\epsilon>0$, which is (\ref{comp}).
\end{proof}

\subsection*{Incomparable radii}
We now want to estimate $\sum_{k>k^{\prime}>N(u)}|\left<G_{u, k^{\prime}}, G_{u, k}\right>|$. Our estimate will be much better than in the comparable radii case. In view of (\ref{goal1}), we will in fact prove the following.
\begin{lemma}\label{inc}
Let $\epsilon>0$. For the choice $N(u)=100\epsilon^{-1}\log_2(2+u)$, we have
\begin{align}\label{i5}
\sum_{k>k^{\prime}>N(u)}|\left<G_{u, k^{\prime}}, G_{u, k}\right>|\lesssim_{\epsilon}\sum_k2^{2k}\#\mathcal{E}_k.
\end{align}
\end{lemma}

Fix $u$ and $k$. Similar to the case of comparable radii, the first step is to cover $\mathcal{E}_{k}(u)$ by a collection $\mathcal{Q}_{u, k}$ of almost-disjoint cubes $Q$ of sidelength $2^{k+5}$. By the almost-disjointness of the cubes, is enough to estimate $|\left<G_{u, k^{\prime}}, G_{u, k}\right>|$ when we restrict our points in $\mathcal{E}_k(u)$ and $\mathcal{E}_{k^{\prime}}(u)$ to points in a fixed $Q^{\ast}$ and get an estimate in terms of $\#(\mathcal{E}_k\cap Q^{\ast})$, after which we may sum in $Q\in\mathcal{Q}_{u, k}$. So fix such a cube $Q$, and let $N_{R, Q, k}$ denote the cardinality of $(\mathcal{E}_k\cap Q^{\ast})_R$ and for a fixed $k^{\prime}$, let $N_{R, Q, k^{\prime}}$ denote the cardinality of $(\mathcal{E}_{k^{\prime}}\cap Q^{\ast})_R$. Similarly, let $N_{Y, Q, k}$ denote the cardinality of $(\mathcal{E}_k\cap Q^{\ast})_Y$ and for a fixed $k^{\prime}$, let $N_{Y, Q, k^{\prime}}$ denote the cardinality of $(\mathcal{E}_{k^{\prime}}\cap Q^{\ast})_Y$. Next, we prove a lemma that plays a role similar to Lemma \ref{rand2} in the comparable radii case.

\begin{lemma}\label{ink}
For each $Q\in\mathcal{Q}_{u, k}$, we have the estimates
\end{lemma}
\begin{multline}\label{i1}
\sum_{(Y, R)\in\mathcal{E}_k(u)\cap Q^{\ast}}\sum_{(y, r)\in\mathcal{E}_{k^{\prime}}(u)\cap Q^{\ast}}|\left<F_{Y, R}, F_{y, r}\right>|
\\
\lesssim R^{2}\#(\mathcal{E}_k\cap Q^{\ast})u(N_{R, Q, k^{\prime}})^{-1}
\end{multline}
and 
\begin{multline}\label{i2}
\sum_{(Y, R)\in\mathcal{E}_k(u)\cap Q^{\ast}}\sum_{(y, r)\in\mathcal{E}_{k^{\prime}}(u)\cap Q^{\ast}}|\left<F_{Y, R}, F_{y, r}\right>|
\\
\lesssim N_{R, Q, k^{\prime}}(\#(\mathcal{E}_k\cap Q^{\ast}))2^{k}(k\log(u))\max(u^{5/6}2^{5k/6}, u2^{k/2}).
\end{multline}

\begin{proof}[Proof of Lemma \ref{ink}]
We will first prove (\ref{i1}), which will be a good estimate in the case that $N_{R, Q, k^{\prime}}$ is large. For each $(Y, R)\in(\mathcal{E}_k(u)\cap Q^{\ast})$ we need only consider $y\in(\mathcal{E}_{k^{\prime}}(u)\cap Q^{\ast})_Y$ lying in an annulus of width $2^{k^{\prime}+5}$ built upon the sphere of radius $R$ centered at $Y$ in $\mathbb{R}^3$. Cover the intersection of this annulus with $(\mathcal{E}_{k^{\prime}}(u)\cap Q^{\ast})_Y$ by a collection $\mathcal{C}$ of $\lesssim R^{2}2^{-2k^{\prime}}$ $3$-dimensional cubes $C$ of sidelength $2^{k^{\prime}+3}$ in $\mathbb{R}^3$ such that each $C\cap(\mathcal{E}_{k^{\prime}}(u)\cap Q^{\ast})_Y$ is nonempty. For each $C\in\mathcal{C}$, let $\tilde{C}$ denote the $4$-dimensional cube $\tilde{C}=C\times [2^{k^{\prime}}-2^{k^{\prime}+2}, 2^{k^{\prime}}+2^{k^{\prime}+2}]$, and let $\tilde{\mathcal{C}}$ denote the corresponding collection of cubes $\tilde{C}$. Now note that ${C}\cap(\mathcal{E}_{k^{\prime}}(u)\cap Q^{\ast})_Y$ nonempty implies that $(\tilde{C}\cap\mathcal{E}_{k^{\prime}}\cap Q^{\ast})_R=(\mathcal{E}_{k^{\prime}}\cap Q^{\ast})_R$, and also that $\#(\tilde{C}\cap\mathcal{E}_{k^{\prime}})\lesssim u2^{k^{\prime}}$, and hence by the product structure of $\tilde{C}\cap\mathcal{E}_{k^{\prime}}\cap Q^{\ast}$,
\begin{multline}\label{cub}
\#((\tilde{C}\cap\mathcal{E}_{k^{\prime}}\cap Q^{\ast})_Y)\lesssim \#(\tilde{C}\cap\mathcal{E}_{k^{\prime}})(\#(\tilde{C}\cap\mathcal{E}_{k^{\prime}}\cap Q^{\ast})_R)^{-1}
\\
\lesssim u2^{k^{\prime}}(N_{R, Q, k^{\prime}})^{-1}.
\end{multline}
\indent
Next, note that for a fixed $Y\in(\mathcal{E}_k\cap Q^{\ast})_Y$, a fixed $R\in(\mathcal{E}_k\cap Q^{\ast})_R$, and a fixed $y\in(\mathcal{E}_{k^{\prime}}\cap Q^{\ast})_Y$, Lemma \ref{bessel} gives rapid decay for $|\left<F_{Y, R}, F_{y, r}\right>|$ as $r$ moves away from two possible values of $r^{\prime}$, that is, when $r$ moves far away from $r^{\prime}=R-|Y-y|$ and $r^{\prime}=|Y-y|-R$. For these values of $r^{\prime}$ we have $|\left<F_{Y, R}, F_{y, r^{\prime}}\right>|\lesssim 2^{k^{\prime}}$. Using (\ref{cub}) and our bound on the size of the collection $\mathcal{C}$, we thus have
\begin{multline*}
\sum_{(Y, R)\in\mathcal{E}_k(u)\cap Q^{\ast}}\sum_{(y, r)\in\mathcal{E}_{k^{\prime}}(u)\cap Q^{\ast}}|\left<F_{Y, R}, F_{y, r}\right>|
\\
\lesssim\sum_{(Y, R)\in\mathcal{E}_k\cap Q^{\ast}}\bigg(\sum_{\tilde{C}\in\tilde{\mathcal{C}}}\bigg(\sum_{(y, r)\in\mathcal{E}_{k^{\prime}}\cap Q^{\ast}\cap\tilde{C}}|\left<F_{Y, R}, F_{y, r}\right>|\bigg)\bigg)
\\
\lesssim\sum_{(Y, R)\in\mathcal{E}_k\cap Q^{\ast}}\bigg(\sum_{\tilde{C}\in\tilde{\mathcal{C}}}\bigg(\sum_{y\in(\mathcal{E}_{k^{\prime}}\cap Q^{\ast}\cap\tilde{C})_Y}\bigg(\sum_{a\in\mathbb{Z}, a\ge 0}
\\
\bigg(\sum_{\substack{r\in(\mathcal{E}_{k^{\prime}}\cap Q^{\ast})_R\\\max(|r^{\prime}-r+|Y-y^{\prime}||, |r^{\prime}+r-|Y-y||)\approx 2^a}}2^{-aN}2^{k^{\prime}}\bigg)\bigg)\bigg)\bigg)
\\
\lesssim R^{2}\#(\mathcal{E}_k\cap Q^{\ast})(N_{R, Q, k^{\prime}})^{-1}u,
\end{multline*}
which is (\ref{i1}).
\newline
\indent
Now we prove (\ref{i2}), which is the estimate that we will use in the case that $N_{R, Q, k^{\prime}}$ is small. This estimate is similar to (\ref{rand3}), and the proof is very similar with only minor modifications, but we give all the details anyways. 
\newline
\indent
By incurring a factor of $N_{R, Q, k}\cdot N_{R, Q, k^{\prime}}$, to estimate 
\begin{align*}
\sum_{(Y, R)\in\mathcal{E}_k(u)\cap Q^{\ast}}\sum_{(y, r)\in\mathcal{E}_{k^{\prime}}(u)\cap Q^{\ast}}|\left<F_{Y, R}, F_{y, r}\right>|,
\end{align*} 
it suffices to estimate for a fixed pair $r_1\in(\mathcal{E}_k\cap Q^{\ast})_R$ and $r_2\in(\mathcal{E}_{k^{\prime}}\cap Q^{\ast})_R$
\begin{align*}
\sum_{(Y, r_1)\in\mathcal{E}_k\cap Q^{\ast}}\sum_{(y, r_2)\in\mathcal{E}_{k^{\prime}}\cap Q^{\ast}}|\left<F_{Y, r_1}, F_{y, r_2}\right>|.
\end{align*}
Similar to the proof of (\ref{rand3}), for $s\ge 0$, let $N_{Y, Q, k}^{\prime}=2^s\le N_{Y, Q, k}$ be a given dyadic number. Fix $t\le 2^{k+10}$, and define $K_{k, k^{\prime}}(Q, s, t)$ to be the number of points $y\in(\mathcal{E}_k(u)\cap Q^{\ast})_Y$ such that there are $\ge N_{Y, Q, k}^{\prime}=2^s$ many points $y^{\prime}\in(\mathcal{E}_{k^{\prime}}\cap Q^{\ast})_Y$ such that $y^{\prime}$ lies in the annulus of inner radius $t$ and thickness $3$ centered at $y$. That is, define
\begin{align*}
K_{k, k^{\prime}}(Q, s, t):=\#\{y\in(\mathcal{E}_k(u)\cap Q^{\ast})_Y: \text{there exists at least }2^s\text{ many points }
\\
y^{\prime}\in(\mathcal{E}_{k^{\prime}}\cap Q^{\ast})_Y\text{ such that }||y^{\prime}-y|-(t+1.5)|\le 1.5\}.
\end{align*}
Also define
\begin{align*}
K_{k, k^{\prime}}^{\ast}(Q, s):=\max_{0\le t\le 2^{k+10}}K_{k, k^{\prime}}(Q, s, t).
\end{align*}
Note that the product structure of $\mathcal{E}$ implies that if both $\mathcal{E}_{k}\cap Q^{\ast}$ and $\mathcal{E}_{k^{\prime}}\cap Q^{\ast}$ are nonempty, then their $\mathcal{Y}$-projections are equal, and so (\ref{Kest2}) implies the bound
\begin{multline}\label{Kest3}
K_{k, k^{\prime}}(Q, s, t)\lesssim
\\
\max\bigg\{\min(u2^kN_{Y, Q, k}^{5/3}2^{-2s}, N_{Y, Q, k}), \min(u2^{k/2}N_{Y, Q, k}2^{-s}, N_{Y, Q, k})\bigg\}.
\end{multline}
Using Lemma \ref{bessel}, we may bound
\begin{multline}\label{b1}
\sum_{\substack{(Y, R)\in(\mathcal{E}_k(u)\cap Q^{\ast})\\(y, r)\in(\mathcal{E}_{k^{\prime}}(u)\cap Q^{\ast})}}|\left<F_{Y, R}, F_{y, r}\right>|
\\
\lesssim \sum_{\substack{r_1\in (\mathcal{E}_k(u)\cap Q^{\ast})_R\\r_2\in (\mathcal{E}_{k^{\prime}}(u)\cap Q^{\ast})_R}}\bigg(\sum_{\substack{Y\in(\mathcal{E}_k(u)\cap Q^{\ast})_Y\\y\in(\mathcal{E}_{k^{\prime}}\cap Q^{\ast})_Y}}|\left<F_{Y, r_1}, F_{y, r_2}\right>|\bigg)
\\
\lesssim 2^{k}\sum_{\substack{r_1\in (\mathcal{E}_k(u)\cap Q^{\ast})_R\\r_2\in (\mathcal{E}_{k^{\prime}}(u)\cap Q^{\ast})_R}}\bigg(\sum_{0\le a\le m+10}\bigg(\sum_{Y\in(\mathcal{E}_k(u)\cap Q^{\ast})_Y}
\\
\sum_{\substack{y\in(\mathcal{E}_{k^{\prime}}\cap Q^{\ast})_Y:\\min_{\pm, \pm}(1+|r_1\pm r_2\pm |y-y^{\prime}||)\approx 2^a}}
2^{-aN}
\bigg)\bigg)
\\
\lesssim 2^{k}\sum_{\substack{r_1\in (\mathcal{E}_k(u)\cap Q^{\ast})_R\\r_2\in (\mathcal{E}_{k^{\prime}}(u)\cap Q^{\ast})_R}}\bigg(\sum_{0\le a\le m+10}2^{-aN}
\\
\times\bigg(\sum_{s\ge 0: 2^s\le 2N_{Y, Q, k}}K_{k, k^{\prime}}^{\ast}(Q, s)2^s\bigg)\bigg)
\\
\lesssim 2^{k}N_{R, Q, k}N_{R, Q, k^{\prime}}\sum_{s\ge 0: 2^s\le 2N_{Y, Q, k}}K_{k, k^{\prime}}^{\ast}(Q, s)2^s.
\end{multline}
Applying (\ref{Kest3}), we have
\begin{multline}\label{mult1}
\sum_{\substack{(Y, R)\in(\mathcal{E}_k(u)\cap Q^{\ast})\\(y, r)\in(\mathcal{E}_{k^{\prime}}(u)\cap Q^{\ast})}}|\left<F_{Y, R}, F_{y, r}\right>|
\\
\lesssim N_{R, Q, k}N_{R, Q, k^{\prime}}2^{k}\sum_{s\ge 0: 2^s\lesssim N_{Y, Q, k}}\max\bigg\{\min(u2^kN_{Y, Q, k}^{5/3}2^{-s}, N_{Y, Q, k}2^s), 
\\
\min(u2^{k/2}N_{Y, Q, k}, N_{Y, Q, k}2^s)\bigg\}.
\end{multline}
Now, note that $u2^kN_{Y, Q}^{5/3}2^{-s}\ge N_{Y, Q}2^s$ if and only if $2^s\le u^{1/2}2^{k/2}N_{Y, Q}^{1/3}$. Also note that $u2^{k/2}N_{Y, Q, k}\ge N_{Y, Q, k}2^s$ if and only if $2^s\le u2^{k/2}$. Thus choosing the better estimate in the term $\min(u2^mN_{Y, Q}^{5/3}2^{-s}, N_{Y, Q}2^s)$ depending on $s$ and the better estimate in the term $\min(u2^{k/2}N_{Y, Q, k}, N_{Y, Q, k}2^s)$ yields that the left hand side of (\ref{mult1}) is bounded by
\begin{align}\label{al1}
N_{R, Q, k}N_{R, Q, k^{\prime}}2^{k}N_{Y, Q, k}\log(N_{Y, Q, k})\max(N_{Y, Q, k}^{1/3}u^{1/2}2^{k/2}, u2^{k/2}).
\end{align}
Using $N_{Y, Q, k}\lesssim u2^{k}$, (\ref{al1}) is bounded by
\begin{multline*}
N_{R, Q, k}N_{R, Q, k^{\prime}}2^{k}N_{Y, Q, k}(k\log(u))\max(u^{5/6}2^{5k/6}, u2^{k/2})
\\
\lesssim N_{R, Q, k^{\prime}}(\#(\mathcal{E}_k\cap Q^{\ast}))2^{k}(k\log(u))\max(u^{5/6}2^{5k/6}, u2^{k/2}),
\end{multline*}
which completes the proof of (\ref{i2}).
\end{proof}

\begin{proof}[Proof of Lemma \ref{inc}]

\indent
Fix $\epsilon>0$, and set $N(u)=100\epsilon^{-1}\log_2(2+u)$. We apply (\ref{i1}) when $N_{R, Q, k^{\prime}}\ge 2^{k^{\prime}\epsilon}$ and (\ref{i2}) when $N_{R, Q, k^{\prime}}\le 2^{k^{\prime}\epsilon}$, and then we sum over $N(u)<k^{\prime}< k$ for $k$ fixed to obtain
\begin{multline}\label{i3}
\sum_{\substack{N(u)<k^{\prime}<k\\k\text{ fixed}}}\,\sum_{(Y, R)\in\mathcal{E}_k(u)\cap Q^{\ast}}\sum_{(y, r)\in\mathcal{E}_{k^{\prime}}(u)\cap Q^{\ast}}|\left<F_{Y, R}, F_{y^{\prime}, r^{\prime}}\right>|
\\
\lesssim_{\epsilon} R^{2}\#(\mathcal{E}_k\cap Q^{\ast})\max(1, \log(u)u^{5/6}2^{-k/6+\epsilon}, \log(u)u2^{-k/2+\epsilon}).
\end{multline}
Next we sum over $Q\in\mathcal{Q}_{u, k}$ and $k>N(u)$ to obtain
\begin{multline}\label{i4}
\sum_k\sum_{Q\in\mathcal{Q}_{u, k}}\sum_{\substack{N(u)<k^{\prime}<k\\k\text{ fixed}}}\,\sum_{(Y, R)\in\mathcal{E}_k(u)\cap Q^{\ast}}\sum_{(y, r)\in\mathcal{E}_{k^{\prime}}(u)\cap Q^{\ast}}|\left<F_{Y, R}, F_{y^{\prime}, r^{\prime}}\right>|
\\
\lesssim_{\epsilon} \sum_k2^{2k}\#\mathcal{E}_k.
\end{multline}
We have thus shown that for the choice $N(u)=100\epsilon^{-1}\log_2(2+u)$, we have
\begin{align*}
\sum_{k>k^{\prime}>N(u)}|\left<G_{u, k^{\prime}}, G_{u, k}\right>|\lesssim_{\epsilon}\sum_k2^{2k}\#\mathcal{E}_k.
\end{align*}
\end{proof}

\subsection*{Putting it together}
Combining (\ref{goal1}), (\ref{comp}) and (\ref{i5}), we have that for every $\epsilon>0$,
\begin{align}
\Norm{G_u}_2^2=\Norm{\sum_kG_{u, k}}_2^2\lesssim_{\epsilon}\log_2(2+u)\sum_k2^{2k}(\#\mathcal{E}_k)u^{11/13+\epsilon}.
\end{align}

This completes the proof of Lemma \ref{L2lemma} and hence the proof of Proposition \ref{mainprop}. Thus we have finished the proof of Theorem \ref{mainthm}. The rest of the paper will be devoted to the (more technical) proof of Theorem \ref{mainthm2}.

\section{Preliminaries and reductions: part II}\label{prelimsec2}

Similarly to Section \ref{prelimsec}, in this section we will collect necessary preliminary results and reductions to prove Theorem \ref{mainthm2}. Much of the proof of Theorem \ref{mainthm2} will be similar to the proof of Theorem \ref{mainthm}, but there are nontrivial additional technical difficulties to the proof of Theorem \ref{mainthm2} that will make the proof more involved. The main reason for this is the fact that Theorem \ref{mainthm2} is a full $L^p$ characterization rather than a restricted strong type $(p, p)$ result, and therefore we cannot simply assume that our discrete sets $\mathcal{E}$ have product structure as we were able to do in the proof of Theorem \ref{mainthm}. 

\subsection*{Discretization and density decomposition of sets}
The first step will be to discretize our problem, and in preparation for this we will first need to introduce some notation. Let $\mathcal{Y}$ be a $1$-separated set of points in $\mathbb{R}^4$ and let $\mathcal{R}$ be a $1$-separated set of radii $\ge 1$. Let $\mathcal{E}\subset\mathcal{Y}\times\mathcal{R}$ be a finite set, and let 
\begin{align*}
u\in\mathcal{U}=\{2^{\nu}, \nu=0, 1, 2,\ldots\}
\end{align*} 
be a collection of dyadic indices. For each $k$, let $\mathfrak{B}_k$ denote the collection of all $5$-dimensional balls of radius $\le 2^k$. For a ball $B$, let $\text{rad}(B)$ denote the radius of $B$. Following \cite{hns}, define:
\begin{align*}
\mathcal{R}_k:=\mathcal{R}\cap [2^k, 2^{k+1}),
\end{align*}
\begin{align*}
\mathcal{E}_k:=\mathcal{E}\cap(\mathcal{Y}\times\mathcal{R}_k),
\end{align*}
\begin{align*}
\wh{\mathcal{E}}_k(u):=\{(y, r)\in\mathcal{E}_k:\,\exists B\in\mathfrak{B}_k\text{ such that }\#(\mathcal{E}_k\cap B)\ge u\,\text{rad}(B)\},
\end{align*}
\begin{align*}
\mathcal{E}_k(u)=\wh{\mathcal{E}}_k(u)\setminus\bigcup_{\substack{u^{\prime}\in\mathcal{U}\\u^{\prime}>u}}\wh{\mathcal{E}}_k(u^{\prime}).
\end{align*}
We will refer to $u$ as the \textit{density} of the set $\mathcal{E}_k(u)$. Note that we have the decomposition
\begin{align*}
\mathcal{E}_k=\bigcup_{u\in\mathcal{U}}\mathcal{E}_k(u).
\end{align*}
Let $\sigma_r$ denote the surface measure on $rS^{3}$, the $3$-sphere centered at the origin of radius $r$. Now fix a smooth, radial function $\psi_0$ which is supported in the ball centered at the origin of radius $1/10$ such that $\wh{\psi_0}$ vanishes to order $40$ at the origin. Let $\psi=\psi_0\ast\psi_0$. For $y\in\mathcal{Y}$ and $r\in\mathcal{R}$, define
\begin{align*}
F_{y, r}=\sigma_r\ast\psi(\cdot-y).
\end{align*}
For a given function $\gamma:\mathcal{Y}\times\mathcal{R}\to\mathbb{C}$ and finite set $\mathcal{E}\subset\mathcal{Y}\times\mathcal{R}$, further define 
\begin{align*}
G_{u, k}^{\gamma, \mathcal{E}}:=\sum_{(y, r)\in\mathcal{E}_k(u)}\gamma(y, r)F_{y, r},
\end{align*}
\begin{align*}
G_{u}^{\gamma, \mathcal{E}}:=\sum_{k\ge 0}G_{u, k}^{\gamma, \mathcal{E}},
\end{align*}
\begin{align*}
G_{k}^{\gamma, \mathcal{E}}:=\sum_{u\in\mathcal{U}}G_{u, k}^{\gamma, \mathcal{E}}.
\end{align*}

\subsection*{The discretized $L^p$ inequality}

We will prove the following proposition, which implies our main result for compactly supported multipliers.
\begin{proposition}\label{mainprop2}
Let $1<p<36/29$. Let $\gamma:\mathcal{Y}\times\mathcal{R}\to\mathbb{C}$ be a function that is a tensor product, i.e., $\gamma(y, r)=\gamma_1(y)\gamma_2(r)$. For each $j\in\mathbb{Z}$, define
\begin{align*}
\mathcal{E}^{\gamma, j}:=\{(y, r)\in\mathcal{Y}\times\mathcal{R}:\,2^j\le|\gamma(y, r)|<2^{j+1}\}.
\end{align*}
\begin{align*}
\mathcal{E}_k^{\gamma, j}:=\{(y, r)\in\mathcal{Y}\times\mathcal{R}:\,r\in\mathcal{R}_k,\, 2^j\le|\gamma(y, r)|<2^{j+1}\}.
\end{align*}
Then
\begin{align}\label{propline2}
\Norm{\sum_{(y, r)\in\mathcal{E}^{\gamma, j}}\gamma(y, r)F_{y, r}}_p^p\lesssim_p 2^{jp}\sum_{l\ge j}2^{(l-j)/5}\sum_k 2^{3k}\#\mathcal{E}_k^{\gamma, l}.
\end{align}
\end{proposition}
Using the dyadic interpolation lemma (Lemma \ref{dyad}), we obtain the following corollary.
\begin{corollary}\label{maincor2}
Let $\gamma:\mathcal{Y}\times\mathcal{R}\to\mathbb{C}$ be a function that is a tensor product, i.e. $\gamma(y, r)=\gamma_1(y)\gamma_2(r)$. Let $h:\mathbb{R}^{5}\to\mathbb{C}$ be a function that is a tensor product, i.e. $h(y, r)=h_1(y)h_2(r)$. Then for $1<p<36/29$, we have
\begin{align}\label{dis2}
\Norm{\sum_{(y, r)\in\mathcal{Y}\times\mathcal{R}}\gamma(y, r)F_{y, r}}_p\lesssim_p\bigg(\sum_{(y, r)\in\mathcal{Y}\times\mathcal{R}}|\gamma(y, r)|^pr^3\bigg)^{1/p}.
\end{align}
Also
\begin{align}\label{cont2}
\Norm{\int_{\mathbb{R}^3}\int_1^{\infty}h(y, r)F_{y, r}\,dr\,dy}_p\lesssim_p\bigg(\int_{\mathbb{R}^3}\int_1^{\infty}|h(y, r)|^pr^3\,dr\,dy\bigg)^{1/p}.
\end{align}
\end{corollary}

\begin{proof}[Proof that Proposition \ref{mainprop2} implies Corollary \ref{maincor2}]
Apply Lemma \ref{dyad} with $F_j=\sum_{(y, r)\in\mathcal{E}^{\gamma, j}}\gamma(y, r)F_{y, r}$, $M^p$ the implied constant from (\ref{propline2}), and 
\begin{align*}
s_j=\sum_{l\ge j}2^{(l-j)/5}\sum_k2^{3k}\#\mathcal{E}_k^{\gamma, l}
\end{align*}
to obtain (\ref{dis2}).
\newline
\indent
Now we prove (\ref{cont2}). Let $y=z+w$ for $z\in\mathbb{Z}^4$ and $w\in Q_0:=[0, 1)^4$ and $r=n+\tau$ for $n\in\mathbb{N}$ and $0\le\tau<1$. By Minkowski's inequality and (\ref{dis2}),
\begin{multline*}
\Norm{\int_{\mathbb{R}^4}\int_1^{\infty}h(y, r)F_{y, r}\,dr\,dy}_p
\\
\lesssim_p\int\int_{Q_0\times [0, 1)}\Norm{\sum_{z\in\mathbb{Z}^4}\sum_{n=1}^{\infty}h_2(n+\tau)h_1(z+w)F_{z+w, n+\tau}}_p\,dw\,d\tau
\\
\lesssim_p\int\int_{Q_0\times [0, 1)}\bigg(\sum_{z\in\mathbb{Z}^4}\sum_{n=1}^{\infty}|h_2(n+\tau)h_1(z+w)|^p(n+\tau)^3\bigg)^{1/p}\,dw\,d\tau
\\
\lesssim_p\bigg(\int_{\mathbb{R}^3}\int_1^{\infty}|h(y, r)|^pr^3\,dr\,dy\bigg)^{1/p},
\end{multline*}
where in the last step we have used H\"{o}lder's inequality.
\end{proof}

\subsection*{Support size estimates vs. $L^2$ inequalities}
Fix a function $\gamma:\mathcal{Y}\times\mathcal{R}\to\mathbb{C}$ that is a tensor product, i.e. $\gamma(y, r)=\gamma_1(y)\gamma_2(r)$, and fix $j\in\mathbb{Z}$. Let 
\begin{align*}
\tilde{\mathcal{E}}^{\gamma, j}:=\{(y, r)\in\mathcal{Y}\times\mathcal{R}:\, 2^{j-5}\le |\gamma(y, r)|\le 2^{j+5}\},
\end{align*}
and recall the density decomposition
\begin{align*}
\tilde{\mathcal{E}}^{\gamma, j}_k=\bigcup_{u\in\mathcal{U}}\tilde{\mathcal{E}}^{\gamma, j}_k(u)
\end{align*}
defined previously. Define a function $\tilde{G}_{u, k}^{\gamma, \mathcal{E}^{\gamma, j}}:\mathcal{Y}\times\mathcal{R}\to\mathbb{C}$ to be the restriction of the function $G_{u, k}^{\gamma, \tilde{\mathcal{E}}^{\gamma, j}}$ to the set $\mathcal{E}^{\gamma, j}_k$, i.e.,
\begin{align*}
\tilde{G}_{u, k}^{\gamma, \mathcal{E}^{\gamma, j}}(y, r)=
\begin{cases}
\\G_{u, k}^{\gamma, \tilde{\mathcal{E}}^{\gamma, j}}(y, r),&\text{if }(y, r)\in\mathcal{E}^{\gamma, j}_k,
\\
0,&\text{if }(y, r)\notin\mathcal{E}^{\gamma, j}_k.
\end{cases}
\end{align*}
Similarly define
\begin{align*}
\tilde{G}_{u}^{\gamma, \mathcal{E}^{\gamma, j}}=\sum_{k\ge 0}\tilde{G}_{u, k}^{\gamma, \mathcal{E}^{\gamma, j}}
\end{align*}
and
\begin{align*}
\tilde{G}_{k}^{\gamma, \mathcal{E}^{\gamma, j}}=\sum_{u\in\mathcal{U}}\tilde{G}_{u, k}^{\gamma, \mathcal{E}^{\gamma, j}}.
\end{align*}
Note that $\tilde{G}_k^{\gamma, \mathcal{E}^{\gamma, j}}=G_k^{\gamma, \mathcal{E}^{\gamma, j}}$, and $\sum_kG_k^{\gamma, \mathcal{E}^{\gamma, j}}$ appears on the left hand side of the inequality in Proposition \ref{mainprop2}.
Similarly to \cite{hns}, we will show that the functions $\tilde{G}_{u, k}^{\gamma, \mathcal{E}^{j, \gamma}}$ either have relatively small support size or satisfy relatively good $L^2$ bounds. We begin with a support size bound which follows immediately from the similar bound in \cite{hns} that improves as the density $u$ increases.
\begin{customlemma}{C}\label{L1lemma2}
For all $u\in\mathcal{U}$, the Lebesgue measure of the support of $\tilde{G}_{u, k}^{\gamma, \mathcal{E}^{\gamma, j}}$ is $\lesssim u^{-1}2^{3k}\#(\bigcup_{l:\,|l-j|\le 10}\mathcal{E}_k^{\gamma, l})$.
\end{customlemma}
We will prove the following $L^2$ inequality which in some sense an improved version of Lemma $3.6$ from \cite{hns}, although the hypotheses are different since it is crucial that we assume that the underlying set is of the form $\mathcal{E}^{\gamma, j}$, i.e. the $\approx 2^j$ level set of some function $\gamma(y, r)=\gamma_1(y)\gamma_2(r)$. This inequality improves as the density $u$ decreases. In \cite{hns}, the analogous $L^2$ inequality proved is 
\begin{align}\label{hnsl22}
\Norm{\tilde{G}_u^{\gamma, \mathcal{E}^{\gamma, j}}}_2^2\lesssim u^{\frac{2}{d-1}}\log(2+u)2^{2j}\sum_k2^{k(d-1)}\#\bigg(\bigcup_{l:\,|l-j|\le 10}\mathcal{E}_k^{\gamma, j}\bigg).
\end{align}
We use geometric methods to improve on (\ref{hnsl22}) in four dimensions, and our argument will rely on Lemma \ref{geomlemma} proved later in Section \ref{geomsec}.

\begin{lemma}\label{L2lemma2}
Let $\mathcal{E}^{\gamma, j}$, $\mathcal{E}^{\gamma, j}_k$, and $\tilde{G}_{u}^{\gamma, \mathcal{E}^{\gamma, j}}$ be as above. Then for every $\epsilon>0$,
\begin{align*}
\Norm{\tilde{G}_u^{\gamma, \mathcal{E}^{\gamma, j}}}_2^2\lesssim_{\epsilon} u^{\frac{11}{18}+\epsilon}2^{2j}\sum_{l\ge j}2^{(l-j)/10}\sum_k2^{3k}\#\mathcal{E}_k^{\gamma, l}.
\end{align*}
\end{lemma}
Combining Lemma \ref{L1lemma2} and Lemma \ref{L2lemma2}, we obtain the following $L^p$ bound.
\begin{lemma}\label{Lplemma2}
For $p\le 2$, for every $\epsilon>0$,
\begin{align*}
\Norm{\tilde{G}_u^{\gamma, \mathcal{E}^{\gamma, j}}}_p\lesssim_{\epsilon, p}u^{-(1/p-29/36-\epsilon)}2^j(\sum_{l\ge j}2^{(l-j)/10}\sum_k2^{3k}\#\mathcal{E}_k^{\gamma, l})^{1/p}.
\end{align*}
\end{lemma}

\begin{proof}[Proof of Lemma \ref{Lplemma2} given Lemma \ref{L1lemma2} and Lemma \ref{L2lemma2}]
By H\"{o}lder's inequality,
\begin{multline*}
\Norm{\tilde{G}_u^{\gamma, \mathcal{E}^{\gamma, j}}}_p\lesssim_p (\text{meas}(\text{supp}(\tilde{G}_u^{\gamma, \mathcal{E}^{\gamma, j}})))^{1/p-1/2}\Norm{\tilde{G}_u^{\gamma, \mathcal{E}^{\gamma, j}}}_2
\\
\lesssim_{\epsilon, p}u^{29/36-1/p+\epsilon}2^j(\sum_{l\ge j}2^{(l-j)/10}\sum_k2^{3k}\#\mathcal{E}_k^{\gamma, l})^{1/p}.
\end{multline*}
\end{proof}

Summing over $u\in\mathcal{U}$, we obtain Proposition \ref{mainprop2}. Thus to prove Proposition \ref{mainprop2} it suffices to prove Lemma \ref{L2lemma2}. One may deduce Theorem \ref{mainthm2} from Corollary \ref{maincor2} in the same way as one deduces Theorem \ref{mainthm} from Corollary \ref{maincor}.

\section{Proof of the $L^2$ inequality: Part II}
We have shown in Section \ref{prelimsec2} that to prove our main result Theorem \ref{mainthm2} it remains to prove Lemma \ref{L2lemma2}, and this section is dedicated to the proof of that lemma. The proof will rely on a geometric lemma about sizes of multiple intersections of three-dimensional annuli, which is stated and proved in Section \ref{geomsec}.

\subsection*{Another preliminary reduction}
Recall that our goal is to estimate the $L^2$ norm of $\tilde{G}_u^{\gamma, \mathcal{E}^{\gamma, j}}=\sum_{k\ge 0}\tilde{G}_{u, k}^{\gamma, \mathcal{E}^{\gamma, j}}$. Let $N(u)$ be a sufficiently large number to be chosen later (it will be some harmless constant depending on $u$ that is essentially $O(\log(2+u))$). We split the sum in $k$ as $\sum_{k\le N(u)}\tilde{G}_{u, k}^{\gamma, \mathcal{E}^{\gamma, j}}+\sum_{k>N(u)}\tilde{G}_{u, k}^{\gamma, \mathcal{E}^{\gamma, j}}$ and apply Cauchy-Schwarz to obtain
\begin{align}\label{goal1}
\Norm{\sum_k\tilde{G}_{u, k}^{\gamma, \mathcal{E}^{\gamma, j}}}_2^2\lesssim N(u)\bigg[\sum_k\Norm{\tilde{G}_{u, k}^{\gamma, \mathcal{E}^{\gamma, j}}}_2^2+\sum_{k>k^{\prime}>N(u)}|\left<\tilde{G}_{u, k^{\prime}}^{\gamma, \mathcal{E}^{\gamma, j}}, \tilde{G}_{u, k}^{\gamma, \mathcal{E}^{\gamma, j}}\right>|\bigg].
\end{align}
We may thus separately estimate $\sum_k\Norm{\tilde{G}_{u, k}^{\gamma, \mathcal{E}^{\gamma, j}}}_2^2$ and 
\begin{align*}
\sum_{k>k^{\prime}>N(u)}|\left<\tilde{G}_{u, k^{\prime}}^{\gamma, \mathcal{E}^{\gamma, j}}, \tilde{G}_{u, k}^{\gamma, \mathcal{E}^{\gamma, j}}\right>|,
\end{align*} 
which divides the proof of the $L^2$ estimate into two cases, the first being the case of ``comparable radii" and the second being the case of ``incomparable radii."

\subsection*{Comparable radii}
We will first estimate $\sum_k\Norm{\tilde{G}_{u, k}^{\gamma, \mathcal{E}^{\gamma, j}}}_2^2$. Our goal will be to prove the following lemma.

\begin{lemma}\label{lem1}
For every $\epsilon>0$,  
\begin{align}\label{comp}
\Norm{\tilde{G}_{u, k}^{\gamma, \mathcal{E}^{\gamma, j}}}_2^2\lesssim_{\epsilon}u^{11/18+\epsilon}2^{2j}\sum_{l\ge j}2^{(l-j)/10}2^{3k}(\#\mathcal{E}_k^{\gamma, l}).
\end{align}
\end{lemma}

Fix $k$ and $u$. As in \cite{hns}, we first observe that for $(y, r), (y^{\prime}, r^{\prime})\in\tilde{\mathcal{E}}_k^{\gamma, j}(u)\cap\mathcal{E}_k^{\gamma, j}$, we have $\left<F_{y, r}, F_{y^{\prime}, r^{\prime}}\right>=0$ unless $|(y, r)-(y^{\prime}, r^{\prime})|\le 2^{k+5}$. To estimate $\Norm{\tilde{G}_{u, k}^{\gamma, \mathcal{E}^{\gamma, j}}}_2^2$ for a fixed $k$, we would thus like to bound
\begin{align*}
2^{2j}\sum_{\substack{(y, r), (y^{\prime}, r^{\prime})\in\tilde{\mathcal{E}}_k^{\gamma, j}(u)\cap\mathcal{E}_k^{\gamma, j}\\ 2^{m}\le|(y, r)-(y^{\prime}, r^{\prime})|\le 2^{m+1}}}|\left<F_{y, r}, F_{y^{\prime}, r^{\prime}}\right>|
\end{align*}
for all $0\le m\le k+4$.
\newline
\indent 
Now fix $m\le k+4$. Let $\mathcal{Q}_{u, j, k, m}$ be a collection of almost disjoint cubes $Q\subset\mathbb{R}^5$ of sidelength $2^{m+5}$ such that $\tilde{\mathcal{E}}_k^{\gamma, j}(u)\cap\mathcal{E}_k^{\gamma, j}\subset\bigcup_{Q\in\mathcal{Q}_{u, k, j, m}}Q$ and so that every $Q$ has nonempty intersection with $\tilde{\mathcal{E}}_k^{\gamma, j}(u)\cap\mathcal{E}_k^{\gamma, j}$. Let $Q^{\ast}$ denote the $2^{5}$-dilate of $Q$ and $\mathcal{Q}_{u, k, j, m}^{\ast}$ the corresponding collection of dilated cubes. Observe that
\begin{multline}\label{rand1}
\Norm{\tilde{G}_{u, k}^{\gamma, \mathcal{E}^{\gamma, j}}}_2^2\lesssim 2^{2j}\sum_{0\le m\le k+4}\bigg(\sum_{\substack{(y, r), (y^{\prime}, r^{\prime})\in\tilde{\mathcal{E}}_k^{\gamma, j}(u)\cap\mathcal{E}_k^{\gamma, j}\\ 2^{m}\le|(y, r)-(y^{\prime}, r^{\prime})|\le 2^{m+1}}}|\left<F_{y, r}, F_{y^{\prime}, r^{\prime}}\right>|
\\
+\sum_{(y, r)\in\tilde{\mathcal{E}}_k^{\gamma, j}(u)\cap\mathcal{E}_k^{\gamma, j}}\Norm{F_{y, r}}_2^2\bigg)
\\
\lesssim 2^{2j}\sum_{0\le m\le k+4}\bigg(\sum_{Q\in\mathcal{Q}_{u, k, j, m}}\bigg(\sum_{\substack{(y, r), (y^{\prime}, r^{\prime})\in(\tilde{\mathcal{E}}_k^{\gamma, j}(u)\cap\mathcal{E}_k^{\gamma, j}\cap Q^{\ast})\\2^{m}\le|(y, r)-(y^{\prime}, r^{\prime})|\le 2^{m+1}}}|\left<F_{y, r}, F_{y^{\prime}, r^{\prime}}\right>|\bigg)
\\
+\sum_{(y, r)\in\tilde{\mathcal{E}}_k^{\gamma, j}(u)\cap\mathcal{E}_k^{\gamma, j}}\Norm{F_{y, r}}_2^2\bigg).
\end{multline} 
\indent 
For each integer $b\in\mathbb{Z}$ define 
\begin{align*}
\mathcal{E}_k^{\gamma, j, b}:=\{(y, r)\in\mathcal{Y}\times\mathcal{R}_k:\,2^{b-3}\le\gamma_1(y)\le2^{b+3}, 
2^{j-b-3}\le\gamma_2(r)\le 2^{j-b+3}\}.
\end{align*}
Note that
\begin{align*}
\mathcal{E}_k^{\gamma, j}\subset\bigcup_{b\in\mathbb{Z}}\mathcal{E}_k^{\gamma, j, b}\subset\tilde{\mathcal{E}}_k^{\gamma, j}.
\end{align*}
Note also that each set $\mathcal{E}_k^{\gamma, j, b}$ is a product, that is a set of the form $Y\times R$ where $Y\subset\mathcal{Y}$ and $R\subset\mathcal{R}$. It follows that $\mathcal{E}_k^{\gamma, j, b}\cap Q$ is a product for any cube $Q\subset\mathbb{R}^{d+1}$.
\newline
\indent
We also define some parameters associated with a fixed $Q\in\mathcal{Q}_{u, k, j, m}$ and $b\in\mathbb{Z}$. Let $N_{R, Q, b}$ be the cardinality of the $\mathcal{R}$-projection of $\mathcal{E}_k^{\gamma, j, b}\cap Q^{\ast}$, i.e.
\begin{align*}
N_{R, Q, b}:=\#((\mathcal{E}_k^{\gamma, j, b}\cap Q^{\ast})_R)=\#\{r: \exists (y, r)\in\mathcal{E}_k^{\gamma, j, b}\cap Q^{\ast}\}.
\end{align*}
Similarly define
\begin{align*}
N_{Y, Q, b}:=\#((\mathcal{E}_k^{\gamma, j, b}\cap Q^{\ast})_Y)=\#\{y: \exists (y, r)\in\mathcal{E}_k^{\gamma, j, b}\cap Q^{\ast}\}.
\end{align*}
We also note the following important observation which we will use repeatedly. Using the definition of the sets $\tilde{\mathcal{E}}_k^{\gamma, j}(u)$ and the fact that for each $b\in\mathbb{Z}$, $\mathcal{E}_k^{\gamma, j, b}$ has product structure, one may see that if $Q\in\mathcal{Q}_{u, k, j, m}$ is such that $\tilde{\mathcal{E}}_k^{\gamma, j}(u)\cap\mathcal{E}_k^{\gamma, j, b}\cap Q^{\ast}$ is nonempty, then 
\begin{align}\label{ess}
|N_{Y, Q, b}\cdot N_{R, Q, b}|\lesssim\#(\mathcal{E}_k^{\gamma, j, b}\cap Q^{\ast})\lesssim \#(\tilde{\mathcal{E}}_k^{\gamma, j}\cap Q^{\ast})\lesssim u2^m.
\end{align}
\indent
Now, we will organize our sets $\mathcal{E}_k^{\gamma, j, b}$ as follows. For a fixed $m$, given $Q\in\mathcal{Q}_{u, k, j, m}$, we would like to group together those $b\in\mathbb{Z}$ for which $\#(\mathcal{E}_k^{\gamma, j, b}\cap Q)$ has essentially equal cardinality and for which the ratio $N_{Y, Q, b}/N_{R, Q, b}$ is essentially equal. For each pair of integers $(c, d)\in\mathbb{Z}^2$, we define
\begin{align*}
\mathcal{B}_{Q, c, d}:=\{b\in\mathbb{Z}:\,2^{c-1}\le\#(\mathcal{E}_k^{\gamma, j, b}\cap Q^{\ast})<2^c,\,2^{d-1}\le N_{Y, Q, b}/N_{R, Q, b}<2^d\}.
\end{align*}
\indent
Now with (\ref{rand1}) in mind, we will prove the following lemma. 

\begin{lemma}\label{rand2}
 For each $Q\in\mathcal{Q}_{u, k, j, m}$ and each quadruple $(c, d, c^{\prime}, d^{\prime})\in\mathbb{Z}^4$, we have the estimates
\begin{multline}\label{rand3}
\sum_{\substack{(y, r)\in\bigcup_{b\in\mathcal{B}_{Q, c, d}}(\mathcal{E}_k^{\gamma, j, b}\cap\tilde{\mathcal{E}}_k^{\gamma, j}(u)\cap Q^{\ast})\\
 (y^{\prime}, r^{\prime})\in\bigcup_{b\in\mathcal{B}_{Q, c^{\prime}, d^{\prime}}}(\mathcal{E}_k^{\gamma, j, b}\cap\tilde{\mathcal{E}}_k^{\gamma, j}(u)\cap Q^{\ast})
\\2^{m}\le|(y, r)-(y^{\prime}, r^{\prime})|\le 2^{m+1}}}|\left<F_{y, r}, F_{y^{\prime}, r^{\prime}}\right>|
\\
\lesssim 2^{\max((c-d)/2, (c^{\prime}-d^{\prime})/2)}(\max(\#\mathcal{B}_{Q, c, d}, \#\mathcal{B}_{Q, c^{\prime}, d^{\prime}}))^22^{\max(c, c^{\prime})}
\\
\times 2^{3(k-m/2)}(m\log(u))\max(u^{5/6}2^{m}, u2^{m/2})
\end{multline}
and
\begin{multline}\label{rand4}
\sum_{\substack{(y, r)\in\bigcup_{b\in\mathcal{B}_{Q, c, d}}(\mathcal{E}_k^{\gamma, j, b}\cap\tilde{\mathcal{E}}_k^{\gamma, j}(u)\cap Q^{\ast})\\
 (y^{\prime}, r^{\prime})\in\bigcup_{b\in\mathcal{B}_{Q, c^{\prime}, d^{\prime}}}(\mathcal{E}_k^{\gamma, j, b}\cap\tilde{\mathcal{E}}_k^{\gamma, j}(u)\cap Q^{\ast})
\\2^{m}\le|(y, r)-(y^{\prime}, r^{\prime})|\le 2^{m+1}}}|\left<F_{y, r}, F_{y^{\prime}, r^{\prime}}\right>|
\\
\lesssim 2^{3(k-m/2)}2^{\max(c, c^{\prime})}(\max(\#\mathcal{B}_{Q, c, d}, \#\mathcal{B}_{Q, c^{\prime}, d^{\prime}}))^2\\\times u2^m(2^{\max((c-d)/2, (c^{\prime}-d^{\prime})/2)})^{-1}.
\end{multline}
\end{lemma}
Notice that (\ref{rand3}) is the better estimate when $2^{\max((c-d)/2, (c^{\prime}-d^{\prime})/2)}$ is small and (\ref{rand4}) is the better estimate when $2^{\max((c-d)/2, (c^{\prime}-d^{\prime})/2)}$ is large. We will use (\ref{rand3}) when $2^{\max((c-d)/2, (c^{\prime}-d^{\prime})/2)}\le u^{1/12}$ and (\ref{rand4}) when $2^{\max((c-d)/2, (c^{\prime}-d^{\prime})/2)}> u^{1/12}$. This yields the following corollary.
\begin{corollary}\label{randcor1}
\begin{align}\label{rand5}
\sum_{\substack{(y, r)\in\bigcup_{b\in\mathcal{B}_{Q, c, d}}(\mathcal{E}_k^{\gamma, j, b}\cap\tilde{\mathcal{E}}_k^{\gamma, j}(u)\cap Q^{\ast})\\
 (y^{\prime}, r^{\prime})\in\bigcup_{b\in\mathcal{B}_{Q, c^{\prime}, d^{\prime}}}(\mathcal{E}_k^{\gamma, j, b}\cap\tilde{\mathcal{E}}_k^{\gamma, j}(u)\cap Q^{\ast})
\\2^{m}\le|(y, r)-(y^{\prime}, r^{\prime})|\le 2^{m+1}}}|\left<F_{y, r}, F_{y^{\prime}, r^{\prime}}\right>|
\lesssim_{\epsilon} I+II,
\end{align}
where
\begin{multline}
I:=2^{3k}2^{\max(c, c^{\prime})}(\max(\#\mathcal{B}_{Q, c, d}, \#\mathcal{B}_{Q, c^{\prime}, d^{\prime}}))^2
\\
\times u^{\epsilon}2^{m\epsilon}\max(u^{11/12}2^{-m/2}, u^{13/12}2^{-m})
\end{multline}
and
\begin{multline}
II:=2^{3k}2^{\max(c, c^{\prime})}(\max(\#\mathcal{B}_{Q, c, d}, \#\mathcal{B}_{Q, c^{\prime}, d^{\prime}}))^2u^{\epsilon}2^{m\epsilon}u^{11/12}2^{-m/2}.
\end{multline}
\end{corollary}
Now note that if $(\max(\#\mathcal{B}_{Q, c, d}, \#\mathcal{B}_{Q, c^{\prime}, d^{\prime}}))>10000m\log(u)$, then for some $l$ such that $l>j+\frac{1}{10}(\max(\#\mathcal{B}_{Q, c, d}, \#\mathcal{B}_{Q, c^{\prime}, d^{\prime}}))$, we have $\#(Q^{\ast}\cap\mathcal{E}_k^{\gamma, l})\gtrsim 2^{\max(c, c^{\prime})}$. Since $2^{\max(c, c^{\prime})}\lesssim\#(Q^{\ast}\cap\mathcal{E}_k^{\gamma, j})\lesssim u2^m$, this implies that 
\begin{align*}
2^{\max(c, c^{\prime})}(\max(\#\mathcal{B}_{Q, c, d}, \#\mathcal{B}_{Q, c^{\prime}, d^{\prime}}))^2\lesssim 2^{(l-j)/10}\#(\mathcal{E}_k^{\gamma, l}\cap Q^{\ast}).
\end{align*}
Thus Corollary \ref{randcor1} implies the following.
\begin{corollary}\label{randcor2}
\begin{align}\label{rand6}
\sum_{\substack{(y, r)\in\bigcup_{b\in\mathcal{B}_{Q, c, d}}(\mathcal{E}_k^{\gamma, j, b}\cap\tilde{\mathcal{E}}_k^{\gamma, j}(u)\cap Q^{\ast})\\
 (y^{\prime}, r^{\prime})\in\bigcup_{b\in\mathcal{B}_{Q, c^{\prime}, d^{\prime}}}(\mathcal{E}_k^{\gamma, j, b}\cap\tilde{\mathcal{E}}_k^{\gamma, j}(u)\cap Q^{\ast})
\\2^{m}\le|(y, r)-(y^{\prime}, r^{\prime})|\le 2^{m+1}}}|\left<F_{y, r}, F_{y^{\prime}, r^{\prime}}\right>|
\lesssim I+II,
\end{align}
where
\begin{multline}
I:=2^{3k}\sum_{l\ge j}2^{(l-j)/10}\#(\mathcal{E}_k^{\gamma, l}\cap Q^{\ast})
\\
\times u^{\epsilon}2^{m\epsilon}\max(u^{11/12}2^{-m/2}, u^{13/12}2^{-m})
\end{multline}
and
\begin{align}
II:=2^{3k}\sum_{l\ge j}2^{(l-j)/10}\#(\mathcal{E}_k^{\gamma, l}\cap Q^{\ast})u^{\epsilon}2^{m\epsilon}
u^{11/12}2^{-m/2}.
\end{align}
\end{corollary}
\indent
By (\ref{ess}) there are $\lesssim m^4\log(u)^4$ quadruples $(c, d, c^{\prime}, d^{\prime})$ for which both $\bigcup_{b\in\mathcal{B}_{Q, c, d}}(\mathcal{E}_k^{\gamma, j, b}\cap\tilde{\mathcal{E}}_k^{\gamma, j}(u)\cap Q^{\ast})$ and $\bigcup_{b\in\mathcal{B}_{Q, c^{\prime}, d^{\prime}}}(\mathcal{E}_k^{\gamma, j, b}\cap\tilde{\mathcal{E}}_k^{\gamma, j}(u)\cap Q^{\ast})$ are nonempty, so Corollary (\ref{randcor2}) implies the following.

\begin{corollary}\label{randcor3}
\begin{align}\label{rand7}
\sum_{\substack{(y, r), (y^{\prime}, r^{\prime})\in\bigcup_{b\in\mathbb{Z}}(\mathcal{E}_k^{\gamma, j, b}\cap\tilde{\mathcal{E}}_k^{\gamma, j}(u)\cap Q^{\ast})
\\2^{m}\le|(y, r)-(y^{\prime}, r^{\prime})|\le 2^{m+1}}}|\left<F_{y, r}, F_{y^{\prime}, r^{\prime}}\right>|
\lesssim I+II,
\end{align}
where
\begin{multline}\label{rand75}
I:=2^{3k}\sum_{l\ge j}2^{(l-j)/10}\#(\mathcal{E}_k^{\gamma, l}\cap Q^{\ast})
\\
\times u^{\epsilon}2^{m\epsilon}\max(u^{11/12}2^{-m/2}, u^{13/12}2^{-m})
\end{multline}
and
\begin{align}\label{rand8}
II:=2^{3k}\sum_{l\ge j}2^{(l-j)/10}\#(\mathcal{E}_k^{\gamma, l}\cap Q^{\ast})u^{\epsilon}2^{m\epsilon}
u^{11/12}2^{-m/2}.
\end{align}
\end{corollary}

\begin{proof}[Proof of Lemma \ref{rand2}]
We will first prove (\ref{rand3}). Fix $b\in\mathcal{B}_{Q, c, d}$ and $b^{\prime}\in\mathcal{B}_{Q, c^{\prime}, d^{\prime}}$. Set 
\begin{align*}
N_{Y, Q, b, b^{\prime}}=\max(N_{Y, Q, b}, N_{Y, Q, b^{\prime}})\approx 2^{\max((c+d)/2, (c^{\prime}+d^{\prime})/2)}.
\end{align*}
It suffices to prove
\begin{multline}\label{rad1}
\sum_{\substack{(y, r)\in\mathcal{E}_k^{\gamma, j, b}\cap\tilde{\mathcal{E}}_k^{\gamma, j}(u)\cap Q^{\ast}\\(y^{\prime}, r^{\prime})\in\mathcal{E}_k^{\gamma, j, b^{\prime}}\cap\tilde{\mathcal{E}}_k^{\gamma, j}(u)\cap Q^{\ast}\\2^m\le |(y, r)-(y^{\prime}, r^{\prime})|\le 2^{m+1}}}|\left<F_{y, r}, F_{y^{\prime}, r^{\prime}}\right>|
\lesssim N_{R, Q, b}N_{R, Q, b^{\prime}}N_{Y, Q, b, b^{\prime}}
\\
\times 2^{3(k-m/2)}(m\log(u))\max(u^{5/6}2^{m}, u2^{m/2}).
\end{multline}
After incurring a factor of $N_{R, Q, b}N_{R, Q, b^{\prime}}$, to estimate the left hand side of (\ref{rad1}) it suffices to estimate for a fixed pair $r_1, r_2$
\begin{align}
\sum_{\substack{(y, r_1)\in(\mathcal{E}_k^{\gamma, j, b}\cap\tilde{\mathcal{E}}_k^{\gamma, j}(u)\cap Q^{\ast})^{\times}\\(y^{\prime}, r_2)\in(\mathcal{E}_k^{\gamma, j, b^{\prime}}\cap\tilde{\mathcal{E}}_k^{\gamma, j}(u)\cap Q^{\ast})^{\times}\\2^m\le |(y, r)-(y^{\prime}, r^{\prime})|\le 2^{m+1}}}|\left<F_{y, r}, F_{y^{\prime}, r^{\prime}}\right>|.
\end{align}
i.e. to restrict $(y, r)$ and $(y^{\prime}, r^{\prime})$ to lie in fixed rows of the product-extensions of $(\mathcal{E}_k^{\gamma, j, b}\cap\tilde{\mathcal{E}}_{k}^{\gamma, j}(u)\cap Q^{\ast})^{\times}$ and $(\mathcal{E}_k^{\gamma, j, b^{\prime}}\cap\tilde{\mathcal{E}}_{k}^{\gamma, j}(u)\cap Q^{\ast})^{\times}$, respectively. (Our estimates will not depend on the particular choice of $r_1$ and $r_2$.)
\newline
\indent
Now, referring to the estimate in Lemma \ref{bessel}, we see that for a fixed $y, r_1, r_2$ we have that $|\left<F_{y, r_1}, F_{y^{\prime}, r_2}\right>|$ decays rapidly as $y^{\prime}$ moves away from the set $\{y^{\prime}:\,|y-y^{\prime}|=|r_1-r_2|\text{ or }|y-y^{\prime}|=r_1+r_2\}$, which is contained in a union of two annuli of thickness $2$ and radii $|r_1-r_2|$ and $r_1+r_2$ centered at $y$. 
\newline
\indent
Let $s\ge 0$, fix $t\le 2^{m+10}$, and define $K_k^{\gamma, j, b, b^{\prime}}(Q, s, t)$ to be the number of points $y\in(\mathcal{E}_k^{\gamma, j, b}\cap\tilde{\mathcal{E}}_k^{\gamma, j}(u)\cap Q^{\ast})_Y$ such that there are $\ge 2^s$ many points $y^{\prime}\in(\mathcal{E}_k^{\gamma, j, b^{\prime}}\cap Q^{\ast})_Y$ such that $y^{\prime}$ lies in the annulus of inner radius $t$ and thickness $3$ centered at $y$. That is, define
\begin{align*}
K_k^{\gamma, j, b, b^{\prime}}(Q, s, t):=\#\{y\in(\mathcal{E}_k^{\gamma, j, b}\cap\tilde{\mathcal{E}}_k^{\gamma, j}(u)\cap Q^{\ast})_Y: \text{there exists at least }
\\2^s\text{ many points }
y^{\prime}\in(\mathcal{E}_k^{\gamma, j, b^{\prime}}\cap Q^{\ast})_Y\text{ such that }||y^{\prime}-y|-(t+1.5)|\le 1.5\}.
\end{align*}
In view of the observation in the previous paragraph, for a given $s$ sufficiently large but smaller than $u2^m$, say $s>m+100$, and a fixed number $t\le 2^{m+10}$, we would like to prove a bound on $K_k^{\gamma, j, b, b^{\prime}}(Q, s, t)$. 
Our bound will depend on $s$ and $m$ but be independent of the choice of $t\le 2^{m+10}$. For this reason, we define the quantity
\begin{align*}
K_k^{\gamma, j, b, b^{\prime}, \ast}(Q, s):=\max_{0 \le t\le 2^{m+10}}K_k^{\gamma, j, b, b^{\prime}}(Q, s, t).
\end{align*}
We will prove
\begin{multline}\label{Kest}
K_k^{\gamma, j, b, b^{\prime}, \ast}(Q, s)\lesssim
\\
\max[u2^{4m/3}N_{Y, Q, b, b^{\prime}}^{5/3}2^{-2s}, u2^{m/2}N_{Y, Q, b, b^{\prime}}2^{-s}], \qquad s>m+100.
\end{multline}
Combining this with the trivial bound $K_k^{\gamma, j, b, b^{\prime}, \ast}(Q, s)\lesssim N_{Y, Q, b, b^{\prime}}$ yields
\begin{multline}\label{Kest2}
K_k^{\gamma, j, b, b^{\prime}, \ast}(Q, s)\lesssim
\\
\max[\min(u2^{4m/3}N_{Y, Q, b, b^{\prime}}^{5/3}2^{-2s}, N_{Y, Q, b, b^{\prime}}), 
\\
\min(u2^{m/2}N_{Y, Q, b, b^{\prime}}2^{-s}, N_{Y, Q, b, b^{\prime}})],
\\s>m+100.
\end{multline}
Note that (\ref{Kest}) gives decay in the number of points $K_k^{\gamma, j, b, b^{\prime}, \ast}(Q, s)$ (i.e. $K_k^{\gamma, j, b, b^{\prime} \ast}(Q, s)\ll N_{Y, Q, b, b^{\prime}}$) if we have that both
\begin{enumerate}
\item$N_{Y, Q, b, b^{\prime}}^{5/3}2^{-2s}u2^{4m/3}\ll N_{Y, Q, b, b^{\prime}}$, i.e. if $2^s\gg N_{Y, Q, b, b^{\prime}}^{1/3}u^{1/2}2^{2m/3}$, and
\item  $N_{Y, Q, b, b^{\prime}}2^{-s}u2^{m/2}\ll N_{Y, Q, b, b^{\prime}}$, i.e. if $2^s\gg u2^{m/2}$.
\end{enumerate}
Using Lemma \ref{bessel}, we may bound
\begin{multline}\label{b1}
\sum_{\substack{(y, r)\in(\mathcal{E}_k^{\gamma, j, b}\cap\tilde{\mathcal{E}}_k^{\gamma, j}(u)\cap Q^{\ast})\\
(y^{\prime}, r^{\prime})\in(\mathcal{E}_k^{\gamma, j, b^{\prime}}\cap\tilde{\mathcal{E}}_k^{\gamma, j}(u)\cap Q^{\ast})
\\2^{m}\le|(y, r)-(y^{\prime}, r^{\prime})|\le 2^{m+1}}}|\left<F_{y, r}, F_{y^{\prime}, r^{\prime}}\right>|
\\
\lesssim \sum_{\substack{r_1\in (\mathcal{E}_k^{\gamma, j, b}\cap\tilde{\mathcal{E}}_k^{\gamma, j}(u)\cap Q^{\ast})_R
\\
r_2\in (\mathcal{E}_k^{\gamma, j, b^{\prime}}\cap\tilde{\mathcal{E}}_k^{\gamma, j}(u)\cap Q^{\ast})_R}}\bigg(\sum_{\substack{y, y^{\prime}\in(\mathcal{E}_k^{\gamma, j, b}\cap\tilde{\mathcal{E}}_k^{\gamma, j}(u)\cap Q^{\ast})_Y\\2^{m}\le|(y, r_1)-(y^{\prime}, r_2)|\le 2^{m+1}}}|\left<F_{y, r_1}, F_{y^{\prime}, r_2}\right>|\bigg)
\\
\lesssim 2^{3(k-m/2)}\sum_{\substack{r_1\in (\mathcal{E}_k^{\gamma, j, b}\cap\tilde{\mathcal{E}}_k^{\gamma, j}(u)\cap Q^{\ast})_R
\\
r_2\in (\mathcal{E}_k^{\gamma, j, b^{\prime}}\cap\tilde{\mathcal{E}}_k^{\gamma, j}(u)\cap Q^{\ast})_R}}
\bigg(\sum_{0\le a\le m+10}\bigg(\sum_{y\in(\mathcal{E}_k^{\gamma, j, b}\cap\tilde{\mathcal{E}}_k^{\gamma, j}(u)\cap Q^{\ast})_Y}
\\
\sum_{\substack{y^{\prime}\in(\mathcal{E}_k^{\gamma, j, b^{\prime}}\cap\tilde{\mathcal{E}}_k^{\gamma, j}(u)\cap Q^{\ast})_Y:\\min_{\pm, \pm}(1+|r_1\pm r_2\pm |y-y^{\prime}||)\approx 2^a}}
2^{-aN}
\bigg)\bigg)
\\
\lesssim 2^{3(k-m/2)}\sum_{\substack{r_1\in (\mathcal{E}_k^{\gamma, j, b}\cap\tilde{\mathcal{E}}_k^{\gamma, j}(u)\cap Q^{\ast})_R
\\
r_2\in (\mathcal{E}_k^{\gamma, j, b^{\prime}}\cap\tilde{\mathcal{E}}_k^{\gamma, j}(u)\cap Q^{\ast})_R}}
\bigg(\sum_{0\le a\le m+10}2^{-aN}
\\
\times\bigg(\sum_{s\ge 0: 2^s\le 2N_{Y, Q, b}}K_k^{\gamma, j, b, b^{\prime}, \ast}(Q, s)2^s\bigg)\bigg)
\\
\lesssim 2^{3(k-m/2)}N_{R, Q, b}N_{R, Q, b^{\prime}}\sum_{s\ge 0: 2^s\le 2N_{Y, Q, b}}K_k^{\gamma, j, b, b^{\prime}, \ast}(Q, s)2^s.
\end{multline}
Assuming (\ref{Kest2}) holds, we have
\begin{multline}\label{longy1}
\sum_{\substack{(y, r)\in(\mathcal{E}_k^{\gamma, j, b}\cap\tilde{\mathcal{E}}_k^{\gamma, j}(u)\cap Q^{\ast})\\
(y^{\prime}, r^{\prime})\in(\mathcal{E}_k^{\gamma, j, b^{\prime}}\cap\tilde{\mathcal{E}}_k^{\gamma, j}(u)\cap Q^{\ast})
\\2^{m}\le|(y, r)-(y^{\prime}, r^{\prime})|\le 2^{m+1}}}|\left<F_{y, r}, F_{y^{\prime}, r^{\prime}}\right>|
\\
\lesssim N_{R, Q, b}N_{R, Q, b^{\prime}}2^{3(k-m/2)}\sum_{s\ge 0: 2^s\lesssim N_{Y, Q, b, b^{\prime}}}\max[\min(u2^{4m/3}N_{Y, Q, b, b^{\prime}}^{5/3}2^{-s}, 
\\
N_{Y, Q, b, b^{\prime}}2^s), 
\min(u2^{m/2}N_{Y, Q, b, b^{\prime}}, N_{Y, Q, b, b^{\prime}}2^s)]
\\
\lesssim N_{R, Q, b}N_{R, Q, b^{\prime}}2^{3(k-m/2)}\bigg[m2^mN_{Y, Q, b, b^{\prime}}+\max\bigg\{\sum_{s\ge 0: 2^s\lesssim N_{Y, Q, b, b^{\prime}}}
\\
\min(u2^{4m/3}N_{Y, Q, b, b^{\prime}}^{5/3}2^{-s}, N_{Y, Q, b, b^{\prime}}2^s),
\\
\sum_{s\ge 0: 2^s\lesssim N_{Y, Q, b, b^{\prime}}}\min(u2^{m/2}N_{Y, Q, b, b^{\prime}}, N_{Y, Q, b, b^{\prime}}2^s)\bigg\}\bigg].
\end{multline}
Now, note that $u2^{4m/3}N_{Y, Q, b, b^{\prime}}^{5/3}2^{-s}\ge N_{Y, Q, b, b^{\prime}}2^s$ if and only if $2^s\le u^{1/2}2^{2m/3}N_{Y, Q, b, b^{\prime}}^{1/3}$. Thus choosing the better estimate in the term $\min(u2^{4m/3}N_{Y, Q, b, b^{\prime}}^{5/3}2^{-s}, N_{Y, Q, b, b^{\prime}}2^s)$ depending on $s$ yields that 
\begin{align*}
\sum_{s\ge 0: 2^s\lesssim N_{Y, Q, b, b^{\prime}}}\min(u2^mN_{Y, Q, b, b^{\prime}}^{5/3}2^{-s}, N_{Y, Q, b, b^{\prime}}2^s)\lesssim u^{1/2}2^{2m/3}N_{Y, Q, b, b^{\prime}}^{4/3}.
\end{align*}
Note that $u2^{m/2}N_{Y, Q, b, b^{\prime}}\ge N_{Y, Q, b, b^{\prime}}2^s$ if and only if $2^s\le u2^{m/2}$. Thus choosing the better estimate in the term $\min(u2^{m/2}N_{Y, Q, b, b^{\prime}}, N_{Y, Q, b, b^{\prime}}2^s)$ depending on $s$ yields that
\begin{align*}
\sum_{s\ge 0: 2^s\lesssim N_{Y, Q, b, b^{\prime}}}\min(u2^{m/2}N_{Y, Q, b, b^{\prime}}, N_{Y, Q, b, b^{\prime}}2^s)\lesssim \log(N_{Y, Q, b, b^{\prime}})N_{Y, Q, b, b^{\prime}}\,u2^{m/2}.
\end{align*}
Using that $N_{Y, Q, b, b^{\prime}}^{1/3}\lesssim u^{1/3}2^{m/3}$, it follows that the left hand side of (\ref{longy1}) is bounded by
\begin{multline}
N_{R, Q, b}N_{R, Q, b^{\prime}}2^{3(k-m/2)}N_{Y, Q, b, b^{\prime}}\bigg[m2^m+\log(N_{Y, Q, b, b^{\prime}})
\\
\times\max(N_{Y, Q, b, b^{\prime}}^{1/3}u^{1/2}2^{2m/3}, u2^{m/2})\bigg]
\\
\lesssim N_{R, Q, b}N_{R, Q, b^{\prime}}2^{3(k-m/2)}N_{Y, Q, b, b^{\prime}}(m\log(u))\max(u^{5/6}2^{m}, u2^{m/2})
\\
\lesssim \max(N_{R, Q, b}, N_{R, Q, b^{\prime}})2^{\max(c, c^{\prime})}2^{3(k-m/2)}(m\log(u))\max(u^{5/6}2^{m}, u2^{m/2}),
\end{multline}
which proves (\ref{rand3}). This will be a good estimate when $\max(N_{R, Q, b}, N_{R, Q, b^{\prime}})$ is small. 
\newline
\indent
Thus to prove (\ref{rand3}) it remains to prove (\ref{Kest}). We will in fact prove (\ref{Kest}) with $K_k^{\gamma, j, b, b^{\prime}, \ast}(Q, s)$ replaced by $K_k^{\gamma, j, b, b^{\prime}}(Q, s, t)$, uniformly in $t\le 2^{m+10}$. Fix $t\le 2^{m+10}$ and let $\alpha=\lceil{\log_2(t)}\rceil$ and cover $(\mathcal{E}_k^{\gamma, j, b}\cap\tilde{\mathcal{E}}_k^{\gamma, j}(u)\cap Q^{\ast})_Y$ by $\lesssim 2^{4(m-\alpha)}$ many $4$-dimensional almost-disjoint balls of radius $2^{\alpha+5}$; denote this collection of balls as $\mathfrak{B}=\{B_i\}$. For each $i$, we define a collection of ``special" points $A_{k, i}^{\gamma, j, b, b^{\prime}}(Q, s, t)$ to be the set of all points $y\in(\mathcal{E}_k^{\gamma, j, b}\cap\tilde{\mathcal{E}}_k^{\gamma, j}(u)\cap Q^{\ast})_Y\cap B_i$ such that there are $\ge 2^s$ many points $y^{\prime}\in(\mathcal{E}_k^{\gamma, j, b^{\prime}}\cap Q^{\ast})_Y$ such that $y^{\prime}$ lies in the annulus of radius $t$ and thickness $3$ centered at $y$. That is, we define
\begin{align*}
A_{k, i}^{\gamma, j, b, b^{\prime}}(Q, s, t):=\{y\in(\mathcal{E}_k^{\gamma, j, b}\cap\tilde{\mathcal{E}}_k^{\gamma, j}(u)\cap Q^{\ast})_Y\cap B_i: \text{ there exist at least }
\\
2^s\text{ many points }
y^{\prime}\in(\mathcal{E}_k^{\gamma, j, b^{\prime}}\cap Q^{\ast})_Y\text{ such that }||y^{\prime}-y|-(t+1.5)|\le 1.5\}.
\end{align*}
Let $K_{k, i}^{\gamma, j, b, b^{\prime}}(Q, s, t)$ denote the cardinality of $A_{k, i}^{\gamma, j, b, b^{\prime}}(Q, s, t)$. Now cover each $B_i$ with $\lesssim 2^{4(\alpha-l)}$ many almost disjoint $4$-dimensional balls $B_{i, \alpha}$ of radius $2^l$ for some $l\le \alpha$. Each such ball contains at most $u2^l$ many points of $A_{k, i}^{\gamma, j, b, b^{\prime}}(Q, s, t)$, so for a fixed $i$ there must be at least $\gtrsim K_{k, i}^{\gamma, j, b, b^{\prime}}(Q, s, t)(u2^l)^{-1}$ many balls $B_{i, j}$ that contain at least one point in $A_{k, i}^{\gamma, j, b, b^{\prime}}(Q, s, t)$. Thus there must be at least $\gtrsim K_{k, i}^{\gamma, j, b, b^{\prime}}(Q, s, t)(u2^l)^{-1}$ many such points in $B_i\cap A_{k, i}^{\gamma, j, b, b^{\prime}}(Q, s, t)$ spaced apart by $\gtrsim 2^l$; call this set $D_{k, i}^{\gamma, j, b, b^{\prime}}(Q, s, t)$. But by Lemma \ref{geomcor}, which we prove later in Section \ref{geomsec} of the paper, the size of three-fold intersections of four-dimensional annuli of radius $t\approx 2^\alpha$ and thickness $3$ spaced apart by $\gtrsim 2^l$ with centers lying a ball of radius $2^{\alpha-5}$ is bounded above by $\lesssim 2^{3(\alpha-l)}2^\alpha$ provided that $l\ge \alpha/2+20$. 
\newline
\indent
It follows that if $l\ge \alpha/2+20$, then for each of these $\approx K_{k, i}^{\gamma, j, b, b^{\prime}}(Q, s, t)(u2^l)^{-1}$ many points $p\in D_{k, i}^{\gamma, j, b, b^{\prime}}(Q, s, t)$, there can be at most 
\begin{align*}
\lesssim K_{k, i}^{\gamma, j, b, b^{\prime}}(Q, s, t)^2(u2^l)^{-2}2^{3(\alpha-l)}2^{\alpha}
\end{align*}
points lying inside the $t$-annulus centered at $p$ that are simultaneously contained in at least two other different $t$-annuli centered at points in $D_{k, i}^{\gamma, j, b, b^{\prime}}(Q, s, t)$. This implies that if $N_{Y, Q, b^{\prime}, i}$ denotes the cardinality of $(\mathcal{E}_k^{\gamma, j, b^{\prime}}\cap Q^{\ast})_Y\cap B_i^{\ast}$ where $B_i^{\ast}=10B_i$, then we have
\begin{align}\label{c1}
N_{Y, Q, b^{\prime}, i}\gtrsim K_{k, i}^{\gamma, j, b, b^{\prime}}(Q, s, t)(u2^l)^{-1}2^s,
\end{align}
which is essentially $2^s$ times the number of points in $D_{k, i}^{\gamma, j, b, b^{\prime}}(Q, s, t)$, provided that $2^s$ is much bigger than the total number of points lying inside a $t$-annulus centered at $p$ that are simultaneously contained in at least two other different $t$-annuli centered at points in $D_{k, i}^{\gamma, j, b, b^{\prime}}(Q, s, t)$, i.e. provided that
\begin{align}\label{c2}
K_{k, i}^{\gamma, j, b, b^{\prime}}(Q, s, t)^2(u2^l)^{-2}2^{3(\alpha-l)}2^\alpha\ll 2^s
\end{align}
and 
\begin{align*}
l\ge \alpha/2+20.
\end{align*} 
Solving for $2^l$ in (\ref{c2}) yields
\begin{align}\label{c3}
2^l\gg K_{k, i}^{\gamma, j, b, b^{\prime}}(Q, s, t)^{2/5}2^{4\alpha/5}u^{-2/5}2^{-s/5}.
\end{align}
Since $s\gg m$, we may choose a minimal $l$ such that 
\begin{align*}
2^l\gg\max(K_{k, i}^{\gamma, j, b, b^{\prime}}(Q, s, t)^{2/5}2^{4\alpha/5}u^{-2/5}2^{-s/5}, 2^{\alpha/2})
\end{align*}
for a sufficiently large implied constant. Substituting into (\ref{c1}) yields
\begin{align}\label{c4}
K_{k, i}^{\gamma, j, b, b^{\prime}}(Q, s, t)\lesssim\max[u2^{4m/3}N_{Y, Q, b^{\prime}, i}^{5/3}2^{-2s}, u2^{m/2}N_{Y, Q, b^{\prime}, i}2^{-s}],
\end{align}
and summing over all $i$ and using the almost-disjointness of the $B_i^{\ast}$ gives
\begin{align}\label{c6}
K_{k}^{\gamma, j, b, b^{\prime}}(Q, s, t)\lesssim\max[u2^{4m/3}N_{Y, Q, b^{\prime}}^{5/3}2^{-2s}, u2^{m/2}N_{Y, Q, b^{\prime}}2^{-s}].
\end{align}
Taking the maximum over all $0\le t\le 2^{m+10}$ proves (\ref{Kest}) and hence also (\ref{rand3}).
\newline
\indent
It remains to prove (\ref{rand4}), which we reproduce again below for convenience.
\begin{multline}
\sum_{\substack{(y, r)\in\bigcup_{b\in\mathcal{B}_{Q, c, d}}(\mathcal{E}_k^{\gamma, j, b}\cap\tilde{\mathcal{E}}_k^{\gamma, j}(u)\cap Q^{\ast})\\
 (y^{\prime}, r^{\prime})\in\bigcup_{b\in\mathcal{B}_{Q, c^{\prime}, d^{\prime}}}(\mathcal{E}_k^{\gamma, j, b}\cap\tilde{\mathcal{E}}_k^{\gamma, j}(u)\cap Q^{\ast})
\\2^{m}\le|(y, r)-(y^{\prime}, r^{\prime})|\le 2^{m+1}}}|\left<F_{y, r}, F_{y^{\prime}, r^{\prime}}\right>|
\\
\lesssim 2^{3(k-m/2)}2^{\max(c, c^{\prime})}(\max(\#\mathcal{B}_{Q, c, d}, \#\mathcal{B}_{Q, c^{\prime}, d^{\prime}}))^2\\\times u2^m(2^{\max((c-d)/2, (c^{\prime}-d^{\prime})/2)})^{-1}.
\end{multline}

This will be a good estimate in the case that $2^{\max((c-d)/2, (c^{\prime}-d^{\prime})/2)}$ is large. Without loss of generality, assume that $c^{\prime}-d^{\prime}\ge c-d$. For a fixed $(y, r)\in Q^{\ast}$ and a fixed $y^{\prime}\in(\mathcal{E}_k^{\gamma, j, b^{\prime}}\cap\tilde{\mathcal{E}}_k^{\gamma, j}(u)\cap Q^{\ast})_Y$, there are at most two values of $r^{\prime}$ away from which $\left<F_{y, r}, F_{y^{\prime}, r^{\prime}}\right>$ decays rapidly. Thus using Lemma \ref{bessel} we may estimate
\begin{multline}\label{b2}
\sum_{\substack{(y, r)\in\bigcup_{b\in\mathcal{B}_{Q, c, d}}(\mathcal{E}_k^{\gamma, j, b}\cap\tilde{\mathcal{E}}_k^{\gamma, j}(u)\cap Q^{\ast})\\
 (y^{\prime}, r^{\prime})\in\bigcup_{b\in\mathcal{B}_{Q, c^{\prime}, d^{\prime}}}(\mathcal{E}_k^{\gamma, j, b}\cap\tilde{\mathcal{E}}_k^{\gamma, j}(u)\cap Q^{\ast})
\\2^{m}\le|(y, r)-(y^{\prime}, r^{\prime})|\le 2^{m+1}}}|\left<F_{y, r}, F_{y^{\prime}, r^{\prime}}\right>|
\\
\lesssim\sum_{0\le a\le m+10}\bigg(\sum_{(y, r)\in\bigcup_{b\in\mathcal{B}_{Q, c, d}}(\mathcal{E}_k^{\gamma, j, b}\cap\tilde{\mathcal{E}}_k^{\gamma, j}(u)\cap Q^{\ast})}\bigg(\sum_{y^{\prime}\in\bigcup_{b\in\mathcal{B}_{Q, c^{\prime}, d^{\prime}}}(\mathcal{E}_k^{\gamma, j, b}\cap\tilde{\mathcal{E}}_k^{\gamma, j}(u)\cap Q^{\ast})_Y}
\\
\bigg(\sum_{\substack{r^{\prime}\in\bigcup_{b\in\mathcal{B}_{Q, c^{\prime}, d^{\prime}}}(\mathcal{E}_k^{\gamma, j, b}\cap\tilde{\mathcal{E}}_k^{\gamma, j}(u)\cap Q^{\ast})_R\\2^m\le |(y, r)-(y^{\prime}, r^{\prime})|\le 2^{m+1}\\ \min_{\pm, \pm}(1+|r\pm r^{\prime}\pm |y-y^{\prime}||)\approx 2^a}}2^{-Na}2^{3(k-m/2)}\bigg)\bigg)\bigg)
\\
\lesssim 2^{3(k-m/2)}(\max(\#\mathcal{B}_{Q, c, d}, \#\mathcal{B}_{Q, c^{\prime}, d^{\prime}}))^2\max_{b\in\mathcal{B}_{Q, c, d}}(\#(\mathcal{E}_k^{\gamma, j, b}\cap\tilde{\mathcal{E}}_k^{\gamma, j}(u)\cap Q^{\ast}))
\\
\times \max_{b^{\prime}\in\mathcal{B}_{Q, c, d}}(\#(\mathcal{E}_k^{\gamma, j, b}\cap\tilde{\mathcal{E}}_k^{\gamma, j}(u)\cap Q^{\ast})_Y)
\\
\lesssim 2^{3(k-m/2)}(\max(\#\mathcal{B}_{Q, c, d}, \#\mathcal{B}_{Q, c^{\prime}, d^{\prime}}))^22^{\max(c, c^{\prime})}u2^m(2^{(c^{\prime}-d^{\prime})/2})^{-1},
\end{multline}
and the proof of (\ref{rand4}) is complete.
\end{proof}

We will now use Lemma \ref{rand2} to prove Lemma \ref{lem1}.

\begin{proof}[Proof of Lemma \ref{lem1}]
Fix an $a>0$ to be determined later. Similar to \cite{hns}, we split $\tilde{G}_{u, k}^{\gamma, \mathcal{E}^{\gamma, j}}=\sum_{\mu}\tilde{G}_{u, k, \mu}^{\gamma, \mathcal{E}^{\gamma, j}}$, where for each positive integer $\mu$ we set
\begin{align*}
I_{k, \mu}=[2^k+(\mu-1)u^a, 2^k+\mu u^a),
\end{align*}
\begin{align*}
\mathcal{E}_{k, \mu}=\mathcal{Y}\times I_{k, \mu},
\end{align*}
\begin{align*}
\tilde{G}_{u, k, \mu}^{\gamma, \mathcal{E}^{\gamma, j}}=\sum_{(y, r)\in\mathcal{E}_{k, \mu}\cap\mathcal{E}_k^{\gamma, j}\cap\tilde{\mathcal{E}}_k^{\gamma, j}(u)}\gamma(y, r)F_{y, r},
\end{align*}
and
\begin{align*}
\tilde{G}_{u, k, \mu, r}^{\gamma, \mathcal{E}^{\gamma, j}}=\sum_{y:\,(y, r)\in\mathcal{E}_{k, \mu}\cap\mathcal{E}_k^{\gamma, j}\cap\tilde{\mathcal{E}}_k^{\gamma, j}(u)}\gamma(y, r)F_{y, r}.
\end{align*}
We have
\begin{align}\label{split2}
\Norm{\tilde{G}_{u, k}^{\gamma, \mathcal{E}^{\gamma, j}}}_2^2\lesssim\Norm{\sum_{\mu}\tilde{G}_{u, k, \mu}^{\gamma, \mathcal{E}^{\gamma, j}}}_2^2\lesssim\sum_{\mu}\Norm{\tilde{G}_{u, k, \mu}^{\gamma, \mathcal{E}^{\gamma, j}}}_2^2+\sum_{\mu^{\prime}>\mu+10}|\left<\tilde{G}_{u, k, \mu^{\prime}}^{\gamma, \mathcal{E}^{\gamma, j}}, G_{u, k, \mu}^{\gamma, \mathcal{E}^{\gamma, j}}\right>|.
\end{align}
By Cauchy-Schwarz,
\begin{align*}
\Norm{\tilde{G}_{u, k, \mu}^{\gamma, \mathcal{E}^{\gamma, j}}}_2^2\lesssim u^a\sum_{r\in\mathcal{I}_{k, \mu}\cap\mathcal{R}}\Norm{\tilde{G}_{u, k, \mu, r}^{\gamma, \mathcal{E}^{\gamma, j}}}_2^2.
\end{align*}
Write
\begin{align*}
\tilde{G}_{u, k, \mu, r}^{\gamma, \mathcal{E}^{\gamma, j}}=\bigg(\sum_{y:\,(y, r)\in\mathcal{E}_{k, \mu}\cap\mathcal{E}_k^{\gamma, j}\cap\tilde{\mathcal{E}}_k^{\gamma, j}(u)}\gamma(y, r)\psi_0(\cdot-y)\bigg)\ast(\sigma_r\ast\psi_0).
\end{align*}
By the Fourier decay of $\sigma_r$ and the order of vanishing of $\psi_0$ at the origin, we have
\begin{align*}
\Norm{\widehat{\sigma_r}\widehat{\psi}_0}_{\infty}\lesssim r^{3/2}.
\end{align*}
Since the square of the $L^2$ norm of $\sum_{y:\,(y, r)\in\mathcal{E}_{k, \mu}\cap\mathcal{E}_k^{\gamma, j}\cap\tilde{\mathcal{E}}_k^{\gamma, j}(u)}\gamma(y, r)\psi_0(\cdot -y)$ is $\lesssim 2^{2j}\#\{y\in\mathcal{Y}:\,(y, r)\in\mathcal{E}_{k, \mu}\cap\mathcal{E}_k^{\gamma, j}\cap\tilde{\mathcal{E}}^{\gamma, j}(u)\}$, we have
\begin{align}\label{aest}
\sum_{\mu}\Norm{\tilde{G}_{u, k, \mu}^{\gamma, \mathcal{E}^{\gamma, j}}}_2^2\lesssim u^a\sum_{\mu}\sum_{r\in\mathcal{I}_{k, \mu}\cap\mathcal{R}}\Norm{\tilde{G}_{u, k, \mu, r}^{\gamma, \mathcal{E}^{\gamma, j}}}_2^2\lesssim 2^{2j}u^a2^{3k}\#\mathcal{E}_k^{\gamma, j}.
\end{align}
By (\ref{split2}), it remains to estimate $\sum_{\mu^{\prime}>\mu+10}|\left<\tilde{G}_{u, k, \mu^{\prime}}^{\gamma, \mathcal{E}^{\gamma, j}}, \tilde{G}_{u, k, \mu}^{\gamma, \mathcal{E}^{\gamma, j}}\right>|$. 
\newline
\indent
Note that we have the bound
\begin{multline}
\sum_{\mu^{\prime}>\mu+10}|\left<\tilde{G}_{u, k, \mu^{\prime}}^{\gamma, \mathcal{E}^{\gamma, j}}, \tilde{G}_{u, k, \mu}^{\gamma, \mathcal{E}^{\gamma, j}}\right>|
\\
\lesssim 2^{2j}\sum_{\substack{m: 2^m\ge u^a\\(y, r), (y^{\prime}, r^{\prime})\in\tilde{\mathcal{E}}_k^{\gamma, j}(u)\cap\mathcal{E}_k^{\gamma, j}\\2^m\le |(y, r)-(y^{\prime}, r^{\prime})|\le 2^{m+1}}}|\left<F_{y, r}, F_{y^{\prime}, r^{\prime}}\right>|
\\
\lesssim 2^{2j}\sum_{m: 2^m\ge u^a}\bigg(\sum_{Q\in\mathcal{Q}_{u, k, j, m}}\bigg(\sum_{\substack{(y, r), (y^{\prime}, r^{\prime})\in Q\cap\tilde{\mathcal{E}}_k^{\gamma, j}(u)\cap\mathcal{E}_k^{\gamma, j}\\2^m\le |(y, r)-(y^{\prime}, r^{\prime})|\le 2^{m+1}}}|\left<F_{y, r}, F_{y^{\prime}, r^{\prime}}\right>|\bigg)\bigg).
\end{multline}
To estimate the inner sum above, we will use Corollary \ref{randcor3}. Summing over all $Q\in\mathcal{Q}_{u, k, j, m}$ and over all $m$ such that $2^m\ge u^a$ we have
\begin{align}\label{onetwo}
\sum_{\mu^{\prime}>\mu+10}|\left<\tilde{G}_{u, k, \mu^{\prime}}^{\gamma, \mathcal{E}^{\gamma, j}}, \tilde{G}_{u, k, \mu}^{\gamma, \mathcal{E}^{\gamma, j}}\right>|\lesssim_{\epsilon} 2^{2j}( I+II),
\end{align}
where
\begin{align}\label{b3}
I:= 2^{3k}\sum_{l\ge j}2^{(l-j)/10}(\#\mathcal{E}_k^{\gamma, l})u^{\epsilon}\max(u^{11/12-a/2}, u^{13/12-a})
\end{align}
and 
\begin{align}\label{b4}
II:=2^{3k}\sum_{l\ge j}2^{(l-j)/10}(\#\mathcal{E}_k^{\gamma, l})u^{\epsilon}u^{11/12-a/2}.
\end{align}
Combining (\ref{split2}), (\ref{aest}) and (\ref{onetwo}), we thus have the estimate
\begin{align*}
\Norm{\tilde{G}_{u, k}^{\gamma, \mathcal{E}^{\gamma, j}}}_2^2\lesssim_{\epsilon}2^{2j}2^{3k}\sum_{l\ge j}2^{(l-j)/10}(\#\mathcal{E}_k^{\gamma, l})\bigg[u^a
+u^{11/12-a/2+\epsilon}+u^{13/12-a+\epsilon}\bigg].
\end{align*}
Choose $a=11/18$ to obtain
\begin{align*}
\Norm{\tilde{G}_{u, k}^{\gamma, \mathcal{E}^{\gamma, j}}}_2^2\lesssim_{\epsilon}2^{2j}2^{3k}\sum_{l\ge j}2^{(l-j)/10}(\#\mathcal{E}_k^{\gamma, l})u^{11/18+\epsilon}
\end{align*}
for every $\epsilon>0$, which is (\ref{comp}).
\end{proof}

\subsection*{Incomparable radii}
We now want to estimate 
\begin{align*}
\sum_{k>k^{\prime}>N(u)}|\left<\tilde{G}_{u, k^{\prime}}^{\gamma, \mathcal{E}^{\gamma, j}}, \tilde{G}_{u, k}^{\gamma, \mathcal{E}^{\gamma, j}}\right>|.
\end{align*} 
Our estimate will be much better than in the comparable radii case. In fact, since $d=4$, we may simply use the estimate proved for incomparable radii in \cite{hns}, which is more than sufficient for our purposes. We restate this estimate using our notation as follows.
\begin{lemma}\label{inc}
Let $\epsilon>0$. For the choice $N(u)=100\epsilon^{-1}\log_2(2+u)$, we have
\begin{align}\label{i5}
\sum_{k>k^{\prime}>N(u)}|\left<\tilde{G}_{u, k^{\prime}}^{\gamma, \mathcal{E}^{\gamma, j}}, G_{u, k}^{\gamma, \mathcal{E}^{\gamma, j}}\right>|\lesssim_{\epsilon}2^{2j}\sum_k2^{3k}\sum_{l:\,|l-j|\le 10}\#\mathcal{E}_k^{\gamma, l}.
\end{align}
\end{lemma}
For a proof of Lemma \ref{inc}, see \cite{hns}.

\subsection*{Putting it together}
Combining (\ref{goal1}), (\ref{comp}) and (\ref{i5}), we have that for every $\epsilon>0$,
\begin{align}
\Norm{\sum_k\tilde{G}_{u, k}^{\gamma, \mathcal{E}^{\gamma, j}}}_2^2=\lesssim_{\epsilon}u^{11/18+\epsilon}\sum_k2^{3k}2^{2j}\sum_{l\ge j}2^{(l-j)/10}(\#\mathcal{E}_k^{\gamma, l}).
\end{align}

This completes the proof of Lemma \ref{L2lemma2} and hence the proof of Proposition \ref{mainprop2}.

\section{Appendix: a geometric lemma}\label{geomsec}
In this section we prove the geometric lemma used in the previous section.

\begin{lemma}\label{geomlemma}
Fix integers $j, l$ with $l\le j$. Let $2^{j-1}\le t\le 2^{j+1}$. Then the size of the intersection of three annuli in $\mathbb{R}^3$ of thickness $4$ and inner radius $t$ such that the distance between the centers of any pair is at least $2^l$ and no greater than $2^j/10$ is $\lesssim 2^{3(j-l)}$, provided that $l\ge j/2+10$.
\end{lemma}

We will use the following basic lemma which gives an estimate on the size of intersections of two-dimensional annuli. This is an immediate corollary of Lemma $3.1$ in \cite{wolff}.
\begin{customlemma}{D}\label{wolfflemma}
Let $A_1$ and $A_2$ be two annuli in $\mathbb{R}^2$ of thickness $1$ built upon circles $C_1$ and $C_2$ of radius $R$, and let $d$ denote the distance between the centers of $C_1$ and $C_2$. If $d\le R/5$, then $A_1\cap A_2$ is contained in the $10$-neighborhood of an arc of $C_1$ of length $\lesssim R/d$.
\end{customlemma}

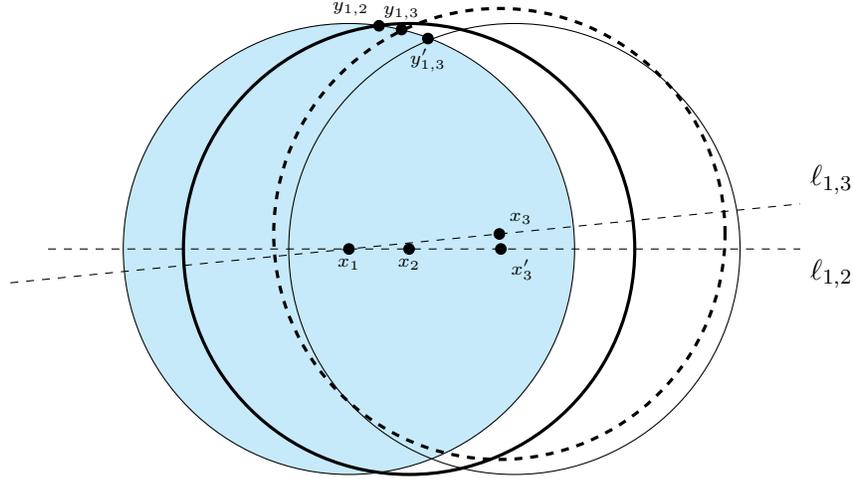
\begin{figure}
\begin{tikzpicture}
\draw[fill=cyan!20!white] (-2,2) circle (3cm);

 \draw[very thick] (-1.2, 2) circle (3cm);

 \filldraw (-2, 2) circle (2pt)  node[anchor=north] {\tiny$x_1$};
  \filldraw (-1.2, 2) circle (2pt) node[anchor=north ] {\tiny$x_2$};
    \filldraw (0, 2.2) circle (2pt) node[anchor=south west] {\tiny$x_3$};

        \filldraw (.02, 2) circle (2pt) node[anchor=north west] {\tiny$x_3^{\prime}$};
 \draw[very thick, dashed] (0, 2.2) circle (3cm);
  \draw(0.2, 2.0) circle (3cm);
 \draw[dashed] (-6,2) -- (4,2) node[anchor=north west]{$\ell_{1, 2}$}; 
 \draw[dashed] (-6.5, 1.55) -- (4, 2.6) node[anchor=south west]{$\ell_{1, 3}$};

 \filldraw (-1.6, 4.97) circle (2pt)  node[anchor=south east] {\tiny$y_{1, 2}$};
  \filldraw (-1.3, 4.92) circle (2pt)  node[anchor=south ] {\tiny$y_{1, 3}$};
    \filldraw (-.95, 4.8) circle (2pt)  node[anchor=north ] {\tiny$y_{1, 3}^{\prime}$};

\end{tikzpicture}
\caption{The circles $C_1, C_2$, $C_3$, and $C_3^{\prime}$ in the plane $P$, from the proof of Lemma \ref{geomlemma}. The shaded-in circle is $C_1$, the thick circle is $C_2$, the dashed circle is $C_3$, and the remaining circle is $C_3^{\prime}$.}
\label{fig1}
\end{figure}

\begin{proof}[Proof of Lemma \ref{geomlemma}]
Let $A_1, A_2, A_3$ denote the three annuli. Let $\ell_{1, 2}$ denote the line through the centers of $A_1$ and $A_2$, and let $\ell_{1, 3}$ denote the line through the centers of $A_1$ and $A_3$. Let $P$ be any plane containing both $\ell_{1, 2}$ and $\ell_{1, 3}$. Then $A_1\cap A_2$ is the three dimensional solid formed by rotating the intersection of the two (circular) annuli $A_1\cap P$ and $A_2\cap P$ about the line $\ell_{1, 2}$. Similarly, $A_1\cap A_3$ is the three dimensional solid formed by rotating the intersection of the two (circular) annuli $A_1\cap P$ and $A_3\cap P$ about the line $\ell_{1, 3}$. 
\newline
\indent
Now, by Lemma \ref{wolfflemma}, $A_1\cap A_2\cap P$ is contained in the $10$-neighborhood of two arcs of length $\lesssim 2^{j-l}$ of the circle that $A_1\cap P$ is built upon. Rotating $A_1\cap A_2\cap P$ about the line $\ell_{1, 2}$ to get $A_1\cap A_2$, this implies that $A_1\cap A_2$ is the union of $\lesssim 2^{j-l}$ many $10$-neighborhoods of circular annuli of radius $\lesssim 2^j$ lying in a plane normal to the line $\ell_{1, 2}$. The same holds for $A_1\cap A_3$ with $\ell_{1, 2}$ replaced by $\ell_{1, 3}$. Suppose first that the angle between $\ell_{1, 2}$ and $\ell_{1, 3}$ is $\ge 2^{l-j-3}$, in radians. Then $|A_1\cap A_2\cap A_3|$ is bounded by $\lesssim 2^{2(j-l)}$ times the largest possible size of the intersection of two $10$-neighborhoods of circular annuli, where the first lies in a plane normal to $\ell_{1, 2}$ and the second lies in a plane normal to $\ell_{1, 3}$. One computes that the largest possible size of such an intersection is $\lesssim 2^{j-l}$.
\newline
\indent
It remains to consider the case when the angle between $\ell_{1, 2}$ and $\ell_{1, 3}$ is $< 2^{l-j-3}$, in radians. We now define the following coordinates associated to the lines $\ell_{1, 2}$ and $\ell_{1, 3}$. Let $x_1, x_2, x_3$ denote the centers of $A_1, A_2, A_3$ respectively. For $x\in\mathbb{R}^3$, we define the $\ell_{1, 2}$-coordinate 
\begin{align*}
(x)_{1, 2}=\frac{\left<x-x_1, x_2-x_1\right>}{|x_2-x_1|}.
\end{align*}
Similarly define the $\ell_{1, 3}$-coordinate
\begin{align*}
(x)_{1, 3}=\frac{\left<x-x_1, x_3-x_1\right>}{|x_3-x_1|}.
\end{align*}
By interchanging the order of $A_1, A_2, A_3$, we may assume without loss of generality that $(x_3)_{1, 2}\ge (x_2)_{1, 2}=1$. We will show that $l\ge j/2+10$ implies that $A_1\cap A_2$ and $A_2\cap A_3$ are actually disjoint. Observe that since the angle between $\ell_{1, 2}$ and $\ell_{1, 3}$ is $< 2^{l-j-3}$, we have that $(x_3-x_2)_{1, 2}\ge 2^{l-1}$. Now, let $x_3^{\prime}$ be the closest point on the line $\ell_{1, 2}$ whose distance from $x_1$ is the same as the distance from $x_1$ to $x_3$. Clearly, we also have $(x_3^{\prime}-x_2)_{1, 2}\ge 2^{l-1}$. Let $C_3$ be the circle in $P$ with center at $x_3$ and radius $t$ and let $C_3^{\prime}$ be the circle in $P$ with center at $x_3^{\prime}$ and radius $t$. Then if $y_{1, 3}^{\prime}$ denotes either of the two points in $C_1\cap C_3^{\prime}$ and $y_{1, 2}$ either of the two points in $C_1\cap C_2$, then $(x_3^{\prime}-x_2)_{1, 2}\ge 2^{l-1}$ implies that $(y_{1, 3}^{\prime}-y_{1, 2})_{1, 2}\ge 2^{l-2}$. This is because with respect to the $\ell_{1, 2}$-coordinate, $y_{1, 3}^{\prime}$ lies at the midpoint of $x_1$ and $x_3$ and $y_{1, 2}$ lies at the midpoint of $x_1$ and $x_2$. Note that $C_1\cap C_3$ is the rotation within $P$ of $C_1\cap C_3^{\prime}$ by an angle of $<2^{l-j-3}$ where the rotation is based at $x_1$. This implies that if $y_{1, 3}$ is either of the two points in $C_1\cap C_3$, then $|y_{1, 3}^{\prime}-y_{1, 3}|\le 2^{l-3}$. It follows that $(y_{1, 3}-y_{1, 2})_{1, 2}\ge (y_{1, 3}^{\prime}-y_{1, 2})_{1,2}-|y_{1, 3}^{\prime}-y_{1, 3}|\ge 2^{l-2}-2^{l-3}=2^{l-3}$. 
\newline
\indent
But by Lemma \ref{wolfflemma}, $A_1\cap A_2$ is the rotation in $\mathbb{R}^3$ of a $10$-neighborhood of an arc of $C_1$ of length $\lesssim 2^{j-l}$ that contains $y_{1, 2}$ about $\ell_{1, 2}$, and so $A_1\cap A_2$ lives in the slab $\{z\in\mathbb{R}^3: |(z-y_{1, 2})_{1, 2}|\le 2^{j-l+4}\}$. Similarly, $A_1\cap A_3$ is the rotation in $\mathbb{R}^3$ of a $10$-neighborhood of an arc of $C_1$ of length $\lesssim 2^{j-l}$ that contains $y_{1, 3}$ about $\ell_{1, 3}$, and so $A_1\cap A_3$ lives in the half-infinite slab $\{z\in\mathbb{R}^3: (z-y_{1, 3})_{1, 2}\ge -2^{j-l+4}\}$, and since $l\ge j/2+10$ we have $j-l+4\le l-10$. Since $(y_{1, 3}-y_{1, 2})_{1, 2}\ge 2^{l-3}$, it follows that $A_1\cap A_2$ and $A_2\cap A_3$ are disjoint.

\end{proof}

\begin{corollary}\label{geomcor}
Fix integers $j, l$ with $l\le j$. Let $2^{j-1}\le t\le 2^{j+1}$. Then the size of the intersection of three annuli in $\mathbb{R}^4$ of thickness $4$ and inner radius $t$ such that the distance between the centers of any pair is at least $2^l$ and no greater than $2^j/10$ is $\lesssim 2^{3(j-l)}2^j$, provided that $l\ge j/2+10$.
\end{corollary}

\begin{proof}
Let $P$ be a hyperplane in $\mathbb{R}^4$ containing the centers of the three annuli, and for each $t\in\mathbb{R}$, let $P_t$ be the one-parameter family of hyperplanes with normals parallel to the normal to $P$. For each $t$, the intersection of each annulus with $P_t$ is a three-dimensional annulus of a fixed radius depending on $t$ that is $\lesssim 2^j$ and a fixed width depending on $t$ that is $\lesssim 1$, and with centers spaced apart by $\gtrsim 2^l$. By Lemma \ref{geomlemma}, $A_1\cap A_2\cap A_3\cap P_t$ has size $\lesssim 2^{3(j-l)}$. It follows that $A_1\cap A_2\cap A_3$ has size $\lesssim 2^{3(j-l)}2^j$.
\end{proof}

\end{document}